\documentclass[a4paper]{article}

\usepackage{RR}
\usepackage[french]{babel}

\usepackage{epsfig}
\usepackage{amssymb}
\usepackage{graphics}
\usepackage{hyperref}

\usepackage{mathrsfs}
\usepackage[latin1]{inputenc}
\usepackage{fancyhdr}
\usepackage{textcomp}
\usepackage{enumitem}
\usepackage{lineno}

\usepackage{latexsym} 
\usepackage{amsmath,amssymb,txfonts}
\usepackage{graphicx}
\usepackage[usenames]{color}

\usepackage{wasysym}

\newcommand{\PartIntSup}[1]{\left\lceil #1\right\rceil}
\newcommand{\PartIntInf}[1]{\left\lfloor #1\right\rfloor}

\usepackage{pb-diagram}

\def\lgem{\hbox{l\kern-.08em\raise .7ex\hbox{.}\kern-.08eml}}
\def\Lgem{\hbox{L\kern-.08em\raise .7ex\hbox{.}\kern-.08emL}}

  \renewenvironment{thebibliography}[1]{%
    \begin{oldthebibliography}{#1}%
      \setlength{\parskip}{0ex}%
      \setlength{\itemsep}{2.0ex}%
  }%
  {%
    \end{oldthebibliography}%
  }

\newtheorem{definition}{Definition}[section]
\newtheorem{theorem} {Theorem}[section]
\newtheorem{lemma} {Lemma}[section]
\newtheorem{observation} {Remark}[section]
\newtheorem{proposition}{Proposition}[section]

\newtheorem{corollary}{Corollary}[section]

\newenvironment
{proof}{\begin{trivlist} \item[] {\em \textbf{Proof}: }}{\hfill
$\Box$\\
                       \end{trivlist}}

\def\ben{\begin{enumerate}%
\itemsep 1pt plus 1pt minus 1pt}
\def\een{\end{enumerate}}
\def\bit{\begin{itemize}%
\itemsep 1pt plus 1pt minus 1pt}
\def\eit{\end{itemize}}

\RRtitle{Groupage de Trafic dans les Anneaux Bidirectionnels WDM
\thanks{This work was partially supported by European project IST FET AEOLUS, the
Ministerio de Educaci\'on y Ciencia of Spain, the European Regional Development Fund under project
TEC2005-03575, and the Catalan Research Council under project 2005SGR00256.}}

\RRetitle{Traffic Grooming in Bidirectional WDM Ring Networks
\thanks{This work was partially supported by European project IST FET AEOLUS, the
Ministerio de Educaci\'on y Ciencia of Spain, the European Regional Development Fund under project
TEC2005-03575, and the Catalan Research Council under project 2005SGR00256.}}

\RRauthor{Jean-Claude Bermond \thanks{Mascotte Project - I3S (CNRS/UNS) and INRIA -
Sophia-Antipolis, France} \and Xavier Muñoz \thanks{Graph Theory and Combinatorics group, MA4, UPC,
Barcelona, Spain} \and Ignasi Sau \thanks{Mascotte Project - I3S (CNRS/UNS) and INRIA -
Sophia-Antipolis, France}
\thanks{Graph Theory and Combinatorics group, MA4, UPC, Barcelona, Spain}}

\RRdate{October 2009}

\RRabstract{We study the minimization of ADMs (Add-Drop Multiplexers) in optical WDM bidirectional
rings considering symmetric shortest path routing and all-to-all unitary requests. We precisely
formulate the problem in terms of graph decompositions, and state a general lower bound for all the
values of the grooming factor $C$ and $N$, the size of the ring. We first study exhaustively the
cases $C=1$, $C = 2$, and $C=3$, providing improved lower bounds, optimal constructions for several
infinite families, as well as asymptotically optimal constructions and approximations. We then
study the case $C>3$, focusing specifically on the case $C = k(k+1)/2$ for some $k \geq 1$. We give
optimal decompositions for several congruence classes of $N$ using the existence of some
combinatorial designs. We conclude with a comparison of the cost functions in unidirectional and
bidirectional WDM rings.}

\RRresume{Nous étudions la minimisation du nombre de MIE (Multiplexeurs a insertion extraction (ADM
en anglais) dans les réseaux  optiques WDM à topologie en cycles (anneaux) bidirectionnels avec un
routage des plus courts chemins et symétrique et ce  pour l'ensemble de requêtes All to ALL (toutes
les requêtes unitaires possibles). nous exprimons le problème en termes de decompositions de
graphes et établissons une borne inférieure générale pour toutes les valeurs du coefficient de
groupage  $C$ et pour tout  $N$ (longueur du cycle). Nous étudions d'abord les cas $C=1$, $C = 2$,
et  $C=3$;  pour ces valeurs  nous donnons des bornes inférieures améliorées, des constructions
optimales pour des familles infinies de valeur de $N$ ainsi que des constructions asymptotiquement
optimales ou de bonnes approximations. Nous étudions ensuite les cas $C>3$, fen particulier $C =
k(k+1)/2$ pour  $k \geq 1$. Nous donnons dans ce cas des décompositions optimales pour plusieurs
classes de congruence  de $N$ en utilisant l'existence de "configurations". Nous concluons en
comparant les fonctions de coût pour les anneaux unidirectionnels et bidirectionnels.}

 \RRkeyword{Traffic grooming, SONET ADM, optical WDM network, graph decomposition,
combinatorial designs.}

\RRmotcle{R\'eseaux optiques, SONET, groupage de trafic, ADM, d\'ecomposition des graphes.}

\RRprojet{MASCOTTE} \RRtheme{\THCom} \URSophia

\begin{document}
\RRNo{7080}

\makeRR

\section{Introduction}

\subsection{Background and Motivation}
Optical wavelength division multiplexing (WDM) is today the most promising technology to
accommodate the explosive growth of Internet and telecommunication traffic in wide-area,
metro-area, and backbone networks. Using WDM, the potential
 bandwidth of approximately 50 THz of a fiber can be divided into multiple non-overlapping wavelength or
 frequency channels. Since currently the commercially available optical fibers can support
 over a hundred frequency channels, such a channel has over one gigabit-per-second transmission
 speed. However, the network is usually required to support traffic connections at rates that
 are lower than the full wavelength capacity. In order to save equipment cost
 and improve network performance, it turns out to be very important to aggregate the multiple
 low-speed traffic connections, namely \emph{requests}, into higher speed streams. Traffic grooming is the term used
 to carry out this aggregation, while optimizing the equipment cost.


  Among possible criteria to minimize the equipment cost, one is to minimize the number of wavelengths used to
  route all the requests~\cite{art26, art27}. A better approximation of the true equipment
  cost is to minimize the number of add/drop locations, namely ADMs using SONET terminology, instead of
  the number of wavelengths. This leads to the \emph{grooming problem}, that we state formally later in Section \ref{sec:state}. These two problems are proved to be
different. Indeed, it is known that even for a simple network like the unidirectional
  ring, the number of wavelengths and the number of ADMs cannot
  be simultaneously minimized \cite{art28,art29}.

SONET ring is the most widely used optical network infrastructure today. In these networks, a
communication between a pair of nodes is done via a \emph{lightpath}, and each lightpath uses an
Add-Drop Multiplexer (\emph{ADM}), i.e. an electronic termination, at each of its two endpoints
(but none in the intermediate nodes). If each request uses $\frac{1}{C}$ of the capacity of a
wavelength, then $C$ is said to be the \emph{grooming factor}, i.e. $C$ requests can be aggregated
in the same wavelength through the same link. If two or more lightpaths using the same wavelength
share a common endpoint, then the same ADM might be used for all lightpaths and therefore the
number of ADMs needed could be reduced. Due to this fact, it makes to sense to try to minimize the
total number of ADMs required.

%



\subsection{Previous Work and Our Contribution}

The notion of traffic grooming was introduced in \cite{first} for the ring topology. Since then,
traffic grooming has been widely studied in the literature (cf. \cite{survey3, survey2,surveyzhu}
for some surveys). The problem has been proved to be $\textsc{NP}$-complete for ring networks and
general $C$ \cite{art28}. Hardness results for rings and paths have been proved in~\cite{APS09}.
Many heuristics have been done, but exact solutions have been found only for certain values of $C$
and for the uniform all-to-all traffic case in unidirectional ring and path
topologies~\cite{BeCo06}.

Many versions of the problem can be considered, according for example to the routing, the physical
graph, and the request graph, among others. For example, in \cite{art5,BBC07} the \textsc{Path
Traffic Grooming} problem is studied. If the network topology is a \emph{ring} (which is the case
of SONET rings), we mainly distinguish two cases depending on the routing. The
\textsc{Unidirectional Ring Traffic Grooming} problem has been studied extensively in the
literature. In an unidirectional ring, requests are routed following only one direction in the
cycle. Up to date, the all-to-all case has been completely solved for values of the grooming factor
until $8$ \cite{BeCo06,CHG+08,CGL09,BCC+05,BCLY04}. Also, recently the unidirectional ring with
bounded degree request graph has been studied~\cite{MuSa08,LiSa09}.

In the \textsc{Bidirectional Ring Traffic Grooming} problem, the scenario is quite different. In a
bidirectional ring, requests are routed either clockwise or counterclockwise. This case has been
much less studied than the unidirectional one, due to its higher complexity. There is an important
work providing heuristics for the ring traffic grooming
\cite{Berry,art28,art27,art29,art30,heuristics,art25,LiWa00}, but there was still an important lack
of theoretical analysis of the problem.
Nevertheless, its study has attracted the interest of numerous researchers. For instance, in
\cite{MILP} a MILP formulation of the problem can be found. In~\cite{WCLF00} two lower bounds is
provided for the number of ADMs in a bidirectional ring with traffic grooming, and in~\cite{art12}
another lower bound is proved, regardless of the routing. In \cite{Wa99,CoWa01,CoLi03,WCLF00} tools
from design theory are applied to the bidirectional ring. Their method is based in the idea of
\emph{primitive rings}, which consists roughly in appropriately generating subgraphs of the request
graph inducing unitary load each, and then packing them into sets of at most $C$ subgraphs. Namely,
in~\cite{WCLF00} several heuristics are proposed, the cases $C=2$ and $C=4$ are studied
in~\cite{Wa99}, the case $C=8$ in~\cite{CoWa01}, and the cases $C=4$ and $C=8$ in~\cite{CoLi03}.
Nevertheless, they do not provide general lower bounds and they do not analyze the approximation
ratio of the proposed algorithms. Therefore, the gaps between their solutions and the optimal ones
are unknown.


In this work we  focus on a bidirectional ring with symmetric shortest path routing, and on the
all-to-all case. We begin by formally stating the problem in terms of graph partitioning in
Section~\ref{sec:state}. In Section \ref{sec:LB} we provide lower bounds and compare them with
those existing in the literature. The remainder of the article is devoted to find families of
solutions for certain values of $C$ and $N$. First we solve in Section~\ref{sec:C1} the case $C=1$.
In Section~\ref{sec:C2} we study the case $C=2$, improving the general lower bound and providing a
$\frac{34}{33}$-approximation. In Section~\ref{sec:C3} we tackle the case $C=3$, improving the
lower bound when $N\equiv 3\pmod{4}$ and giving optimal solutions when $N\equiv0,1,4,5\pmod{12}$.
For all other values of $N$ we give asymptotically optimal solutions. In Section~\ref{sec:C>3} we
use design theory to provide optimal solutions when $C$ is of the form $k(k+1)/2$, for some
congruence classes of values of $N$. We also give improved lower bounds when $C$ is not of the form
$k(k+1)/2$. In Section~\ref{sec:comp} we compare unidirectional and bidirectional rings in terms of
minimizing the cost. We conclude the article in Section~\ref{sec:concl}.

\section{Statement of the Problem}
\label{sec:state}

\subsection{Load constraint}
In a graph-theoretical approach, we are given an optical network represented by a directed graph
$G$ on $N$ vertices (in many cases a symmetric one) --~called the \emph{physical graph}~--, for
example a unidirectional ring $\vec{C}_N$ or a
  bidirectional symmetric ring $C_N^*$. We are given also a traffic (or instance) matrix, that
  is a family of connection requests represented by an arc-weighted multidigraph $I$ --~called the \emph{logical} or \emph{request graph}~--
  where the number of arcs from $i$ to $j$ corresponds to the number of
  requests from $i$ to $j$, and the weight of each arc corresponds to the amount of bandwidth used by each request. Here we suppose
  that there is exactly one request from $i$ to $j$ (all-to-all case) and that each request uses the same bandwidth. In that case $I =
  K_N^*$. We also suppose that the bandwidth used by any request is a fraction $1/C$ of the
  available bandwidth of a wavelength. Said otherwise, each wavelength $\omega$ can carry on a
  given arc at most $C$ requests. This positive integer $C$ is called the \emph{grooming factor}.
  For a wavelength $\omega$, we denote by $B_{\omega}$ the set of requests carried by $\omega$. Satisfying a
  request $r$ from $i$ to $j$ consists in finding a dipath $P(r)$ in $G$
  and assigning it a wavelength $\omega$. Note that a wavelength $\omega$ is directed either
  clockwise or counterclockwise, so all the dipaths associated to requests in a same $B_\omega$ are
  directed in the same way.

For a subgraph $B_\omega$ of requests of $I$, we define the \emph{load} of an arc $e$
 of $G$, $L(B_\omega,e)$, as the number of requests which are routed
through $e$, that is
$$
L(B_\omega,e):=|\{P(r); r\in E(B_\omega); e\in P(r)\}|.
$$

Note that if $B_\omega$ is associated to a clockwise
  (resp. counterclockwise) wavelength $\omega$, only the clockwise (resp. counterclockwise) arcs of
  the ring are loaded by $B_\omega$. The constraint given by the grooming factor $C$ means that for each subgraph $B_\omega$ and each
arc $e$, $L(B_\omega,e)$ is at most $C$. In this article we focus on the bidirectional ring
topology with all-to-all unitary requests. Therefore, our problem consists in finding a partition
of $K_N^*$ into subdigraphs $B_\omega$ satisfying the load constraint for $C_N^*$ and such that the
total number of vertices is minimized. We have two choices for routing a request $(i,j)$: either
clockwise or counterclockwise. Although there is no physical constraint imposing it, it is common
for the operators to consider symmetric routings. That is, if the request $(i,j)$ is routed
clockwise, then the request $(j,i)$ is routed counterclockwise. Furthermore it is also common for
the sake of simplicity to use shortest path routing. Therefore we will restrict ourselves to
symmetric shortest path routings. Let us see how the restrictions on the routing affect the
solutions.

\subsection{Constraints on the routing}
\label{sec:constraints_routing} In a ring $C_N^*$ with an odd number of vertices, shortest path
routing implies symmetric routing. But in a ring with an even number of vertices this is not
necessarily the case, as a request of the form $(i,i+\frac{N}{2})$ can be routed via a shortest
path in both directions. Consider for example $N=4$ and $C=2$. If we do not impose symmetric
routing, we can have a solution consisting of the two subdigraphs $B_{\omega_1}$ with the requests
$(0,1)$, $(1,2)$, $(2,3)$, $(3,0)$, $(0,2)$, and $(2,0)$ routed clockwise, and $B_{\omega_2}$ with
the requests $(1,0)$, $(0,3)$, $(3,2)$, $(2,1)$, $(1,3)$, and $(3,1)$ routed counterclockwise.
Altogether we use $8$ ADMs. Suppose now that we further impose symmetric routing, and assume
without loss of generality that the requests $(0,2)$ and $(1,3)$ are routed clockwise. The best we
can do for a $B_\omega$ with 4 vertices is to put 5 requests if $\omega$ is clockwise, namely
$(0,1)$, $(1,2)$, $(2,3)$, $(3,0)$, and at most one of $(0,2)$ and $(1,3)$. The other request out
of $(0,2)$ and $(1,3)$ will need 2 ADMs, so we use a total of 12 ADMs. If we do not use any
$B_\omega$ with 4 vertices, note that a subdigraph with 3 (resp. 2) vertices contains at most 3
requests (resp. 1 request). Therefore to route all the requests we need at least 12 ADMs.

Imposing shortest path routing might increase the number of ADMs of an optimal solution. Consider
for example $N=3$ and $C=3$. With shortest path routing, we need two subdigraphs $B_{\omega_1}$
with the requests $(0,1)$, $(1,2)$, $(2,0)$ and $B_{\omega_2}$ with the requests $(1,0)$, $(2,1)$,
$(0,2)$, for a total of 6 ADMs (each arc of $C_3^*$ is loaded once). Without the constraint of
shortest path routing, we can do it with 3 ADMs, namely with all the requests routed clockwise. In
that case, the requests $(1,0)$, $(2,1)$, and $(0,2)$ are routed via dipaths of length 2 (for
instance, the request $(1,0)$ uses the arcs $(1,2)$ and $(2,0)$). In that case the load of the arcs
(in the clockwise direction) is 3.

We cannot always use shortest path routing and have a minimum load. Indeed, consider the case $C=1$
and a set of 3 requests $(i,j)$, $(j,k)$, and $(k,i)$ forming a triangle. The subdigraph formed by
the 3 requests routed in the same direction has load 1, but there is not reason that the associated
routes are shortest paths. For example, let $N=5$ and $(0,1)$, $(1,2)$, $(2,0)$ be the three
mentioned requests, which we assume to be routed clockwise. If we want a valid solution, then the
request $(2,0)$ is routed via the path $[2,3,4,0]$ of length 3 (and not 2). If we want to use
shortest paths, then these three requests induce load 2, hence they cannot fit together in the same
wavelength. Summarizing, in this example either we use shortest paths and the load is 2 or we get a
solution with load one but not using shortest paths.

\subsection{Symmetric shortest path routing}
In the sequel of the paper we will only consider \textbf{symmetric shortest path routings}. Besides
being a common scenario in telecommunication networks, this assumption also simplifies the problem,
as we can split it into two separate problems, half of the requests being routed clockwise and half
counterclockwise. Each of these two subproblems can be viewed as a grooming problem where $G =
\vec{C}_N$ (the unidirectional cycle) and $I=T_N$, where $T_N$ is a tournament on $N$ vertices,
that is, a complete oriented graph (for each pair of vertices $\{i,j\}$ there is exactly one of the
arcs $(i,j)$ or $(j,i)$).

As we consider shortest path routing, for $N$ odd $T_N$ is unique. But for $N$ even we have two
possibilities for the pairs of the form $\{i, i + \frac{N}{2}\}$: either the arc
$(i,i+\frac{N}{2})$ or $(i+\frac{N}{2},i)$. So the choice of these arcs has to be made. We are now
ready to state precisely our problem.

\begin{center} \fbox{\begin{minipage}{14.5cm}
\noindent \textsc{Our Traffic Grooming Problem}\\
{\bf Input}: A unidirectional cycle $\vec{C}_N$ with vertices $0,\ldots,N-1$, a grooming factor $C$
and a digraph of requests consisting of the tournament $T_N$ with arcs $(i,i+1)$ for $0 \leq i \leq
N-1$ and $1 \leq q \leq \frac{N-1}{2}$, plus if $N$ is even $\frac{N}{2}$ arcs of the form
$(i,i+\frac{N}{2})$, where we cannot have both $(i,i+\frac{N}{2})$ and $(i+\frac{N}{2},i)$ (or said
otherwise, for $N$ even we have one of the
two arcs $(i,i+\frac{N}{2})$ or $(i+\frac{N}{2},i)$ for $0 \leq i \leq \frac{N}{2}-1$).\\
{\bf Output}: A partition of $T_N$ into digraphs $B_\omega$, $1 \leq \omega \leq W$, such that for
each arc $e \in E(\vec{C}_N)$, $L(B_\omega,e)\leq C$.\\
{\bf Objective}: Minimize $\sum_{\omega=1}^W |V(B_\omega)|$. The minimum
  will be denoted $A(C,N)$.
\end{minipage}}\\
\medskip
\end{center}

\begin{observation}
\label{rem:after_statement} Solutions to the original problem can be found by solving the above
problem and using the solution for the counterclockwise requests by reversing the orientation of
the arcs of $\vec{C}_N$ and $T_N$. Therefore, the total number of ADMs for the original problem
--~under the constraints of symmetric shortest path routing~-- is $2A(C,N)$.
\end{observation}

Let us see an example for $N=5$ and $C=1$. Then the following three subdigraphs form a solution
with $10$ ADMs: one with arcs $(0,1),(1,3),(3,0$), another with arcs $(1,2),(2,4),(4,1$), and
another with arcs $(0,2),(2,3),(3,4),(4,0)$. Thus, a solution for the bidirectional ring $C_5^*$
and $I=K_5^*$ needs 20 ADMs.

Let now $N=5$ and $C=2$. We can use the preceding solution or another one with also 10 ADMs with
only two $\vec{C}_5$'s with arcs $(0,2),(1,2),(2,3),(3,4),(4,5)$ and
$(0,2),(2,4),(4,1),(1,3),(3,0)$, the second one inducing load 2. But we can do better, with only 8
ADMs, with one subdigraph with arcs $(1,3),(3,4),(4,1)$, and another one with arcs
$(0,1),(1,2),(0,2),(2,3),(2,4),(3,0),(4,0)$. This latter partition is optimal. In that case, we
need for the bidirectional ring 16 ADMs.

To tackle our problem we will use tools from design theory, similar to those used for the
unidirectional ring and $I = K_N$~\cite{BeCo03,BeCo06}. In particular, it is helpful to use, for a
given $C$, digraphs having a maximum ratio number of arcs over number of vertices (see
Section~\ref{sec:gamma}).

\subsection{Admissible digraphs} Let $B_\omega = (V_\omega, E_\omega)$ be a digraph with $V_\omega =
\{a_0,\ldots,a_{p-1}\}$ involved in a partition of the tournament $T_N$. Note that the edges of
$B_\omega$ belong to $T_N$, so $(a_i,a_j) \in E_\omega$ if and only $d_{\vec{C_N}}(a_i,a_j) \leq
\frac{N}{2}$, where $d_{\vec{C_N}}(a_i,a_j)$ is the distance between $a_i$ and $a_j$ in
$\vec{C_N}$.

A digraph $B_\omega$ is said to be \emph{admissible} if it satisfies the load constraint, that is,
$L(B_\omega,e)\leq C$ for each arc $e \in E(\vec{C}_N)$. A partition of $T_N$ into admissible
subdigraphs is called \emph{valid}. As the paths associated to an arc of $B_\omega$ form a dipath
(an interval) in $\vec{C}_N$, the load is exactly the same as if we consider $B_\omega$ embedded in
a cycle $\vec{C}_p$ with vertex set $0,1,\ldots,p-1$. More precisely, we associate to $B_\omega$
the digraph $B_\omega^p$ with vertices $0,1,\ldots,p-1$ and with $(i,j) \in E(B_\omega^p)$ if and
only if $(a_i,a_j) \in E(B_\omega)$. Hence, to compute the load we will consider digraphs with $p$
vertices and their load in the associated $\vec{C}_p$. Note that it can happen that
$d_{\vec{C_N}}(a_i,a_j) \leq \frac{N}{2}$ but $d_{\vec{C_p}}(i,j) > \frac{p}{2}$, and viceversa.

Figure~\ref{fig:good2}(a) illustrates a digraph $B_\omega$ that is admissible for $N=8$ and $C=2$,
as it induces load 2 in $\vec{C}_8$. Its associated digraph $B_\omega^4$ is shown in
Figure~\ref{fig:good2}(b). Figure~\ref{fig:good2}(c) shows a digraph $B_\omega'$ which has also
$B_\omega$ as associated digraph, but it is not admissible as $(a_3,a_0)$ is not an arc of $T_8$.


\begin{figure}[h!tb]
\vspace{-0.25cm}
\begin{center}
\includegraphics[width=12.8cm]{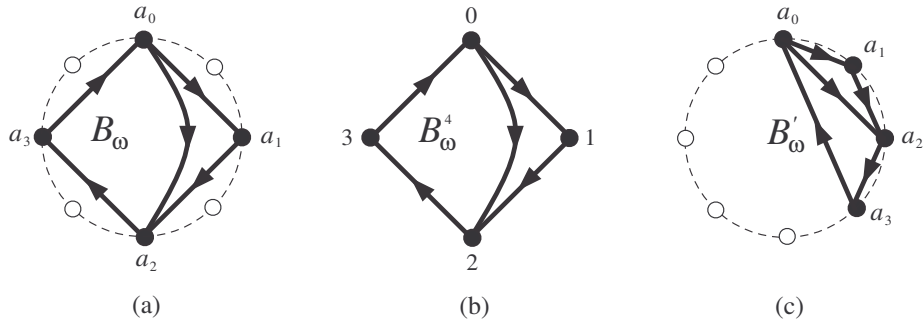}
\vspace{-0.2cm} \caption{(a) Digraph $B_\omega$ admissible for $N=8$ and $C=2$; (b) Its associated
digraph $B_\omega^4$; (c) Non-admissible digraph $B_\omega'$ that has also $B_\omega^4$ as
associated digraph.}
 \label{fig:good2}
\end{center}
\end{figure}

Figure~\ref{fig:good27}(a) shows and admissible digraph for $N=7$ and $C=2$.  Its associated
digraph $B_\omega^5$, which is depicted in Figure~\ref{fig:good27}(b), induces load 2 but the arc
$(1,4)$ is not routed via a shortest path (although the arc $(a_1,a_4)$ was in $B_\omega$).
\begin{figure}[h!tb]
\vspace{-0.25cm}
\begin{center}
\includegraphics[width=8.1cm]{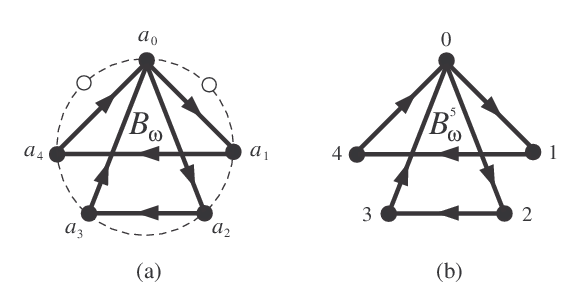}
\vspace{-0.2cm} \caption{(a) Digraph $B_\omega$ admissible for $N=7$ and $C=2$; (b) Its associated
digraph $B_\omega^5$.}
 \label{fig:good27}
\end{center}
\end{figure}

In what follows we will compute the load in the associated digraph, but we will have to be careful
that the arcs of $B_\omega$ are those of $T_N$, as pointed out by the above examples.


\section{Lower Bounds}
\label{sec:LB} In this section we state general lower bounds on the number of ADMs used by any
solution.
\subsection{Equations of the Problem}

Given a valid solution of the problem, let $a_p$ denote the number of subgraphs of the partition
with exactly $p$ nodes, let $A$ denote the total number of ADMs, let $W$ denote the number of
subgraphs of the partition, and let $E_\omega$ be the set of arcs of $B_\omega$. Recall that here
$I=T_N$, which has $\frac{N(N-1)}{2}$ arcs. The following equalities hold:
\begin{eqnarray}
  \label{eq:eqn1} A&= &\sum_{p=2}^N pa_p\\
  \label{eq:eqn2} \sum_{p=2}^N a_p&=&W\\
  \label{eq:eqn3} \sum_{w=1}^W|E_\omega|&=&\frac{N(N-1)}{2}
\end{eqnarray}
\begin{proposition}
\label{lem:W} For $I = T_N$,
$$
W \geq \PartIntSup{\frac{N^2+\alpha}{8C}}\mbox{, where }\alpha= \left\{\begin{array}{cl}
-1,&\mbox{if $N$ is odd}\\
4,&\mbox{if $N\equiv 2 \pmod{4}$}\\
8,&\mbox{if $N\equiv 0 \pmod{4}$}\\
\end{array}\right.
$$
\end{proposition}
\begin{proof}
The set of arcs of $T_N$ of the form $(i,i+q)$, $0\leq q < \frac{N}{2}$, load each arc of the ring
exactly $q$ times. So if $N$ is odd the load of any arc of the ring is
$1+2+\ldots+\frac{N-1}{2}=\frac{N^2-1}{8}$.

If $N$ is even the load due to these arcs is $1+2+\ldots+\frac{N-2}{2}=\frac{N^2-2N}{8}$. We have
to add the load due to arcs of $T_N$ of the form $\left(i,i+\frac{N}{2}\right)$. As there are
$\frac{N}{2}$ such arcs, the total load is $\frac{N^2}{4}$ and so one arc of the ring has load at
least $\frac{N}{4}$.

If $N \equiv 2 \pmod{4}$ that gives a load at least $\PartIntSup{\frac{N}{4}}=\frac{N+2}{4}$, so
one arc has load at least $\frac{N^2-2N}{8}+\frac{N+2}{4}=\frac{N^2+4}{8}$.

If $N \equiv 0 \pmod{4}$ the maximum load due to the arcs $\left(i,i+\frac{N}{2}\right)$ is at
least $\frac{N}{4}$, but in this case we can give a better bound. Indeed, suppose w.l.o.g. that we
have the arc $\left(0,\frac{N}{2}\right)$, and let $j$ be the number of arcs starting in the
interval $[1,\frac{N}{2}-1]$ of the form $\left(i,i+\frac{N}{2}\right)$ with $0<i<\frac{N}{2}$. The
load of the arc $\left(\frac{N}{2}-1,\frac{N}{2}\right)$ of the ring is then $j+1$. As there are
$\frac{N}{2}-1-j$ arcs ending in the interval $[1,\frac{N}{2}-1]$, the load of the arc
$\left(0,1\right)$ is $1+\frac{N}{2}-1-j$. Therefore the sum of the loads of the arcs
$\left(0,1\right)$ and $\left(\frac{N}{2}-1,\frac{N}{2}\right)$ is $\frac{N}{2}+1$, and so one of
these 2 arcs has load $\PartIntSup{\frac{N}{4}+\frac{1}{2}}= \frac{N}{4}+1$. The total load of this
arc is $\frac{N^2-2N}{8}+\frac{N}{4}+1 = \frac{N^2+8}{8}$.

As each subgraph can load one arc at most $C$ times, we obtain the lemma.
\end{proof}

\subsection{The parameter $\gamma(C,p)$}
\label{sec:gamma}

To obtain accurate lower bounds we need to bound the value of $|E_\omega|$ for a digraph with
$|V_\omega|=p$ vertices, satisfying the load constraint (admissible digraph). As we discussed in
the preceding section, we need only to consider the associated digraph embedded in $\vec{C}_p$. To
this end, we introduce the following definitions.
\begin{definition}
 Let $\gamma(C,p)$ be the maximum number of arcs of a digraph $H$ with $p$ vertices such
 that $L(H,e)\leq C$, for every arc $e$ of $\vec{C}_p$.
\end{definition}

\begin{definition}
$$\rho(C)\ =\ \max_{p \geq 2}\left\{\frac{\gamma(C,p)}{p}\right\}.$$
\end{definition}

In~\cite{WCLF00} the authors define two parameters which coincide with the parameters $\gamma(C,p)$
and $\rho(C)$ introduced above. In~\cite{WCLF00} the parameter $\rho(C)$ is called \emph{maximal
ADM efficiency}, and its value is determined, but no closed formula for $\gamma(C,p)$ is given
in~\cite{WCLF00}. Here we give again the value of $\rho(C)$, using different tools, and give the
exact value of $\gamma(C,p)$.

The next proposition shows that, in fact, the maximum number of requests we can groom is attained
by taking those of minimum length. It is worth to mention that this property is not true if the
physical graph is a path, as shown with a counterexample in~\cite{BBC07}.


\begin{proposition}
\label{lem:shortest} Let $C=\frac{k(k+1)}{2}+r$, with $0 \leq r \leq k$. Then
$$
\gamma(C,p)= \left\{\begin{array}{cl} \frac{p(p-1)}{2}&\mbox{, if
}p\leq2k+1\mbox{, or }p = 2k+2 \mbox{ and } r\geq\frac{k+2}{2}\\
kp + 2r -1&\mbox{, if } p = 2k+2 \mbox{ and } 1 \leq r < \frac{k+2}{2}\\
kp+ \PartIntInf{\frac{rp}{k+1}}&\mbox{, otherwise}
\end{array}\right.
$$
The graphs achieving $\gamma(C,p)$ are either the tournament $T_p$ if $p$ is small (namely, if $p
\leq 2k+1$ or $p = 2k+2$ and $r \geq \frac{k+2}{2}$), or subgraphs of a circulant digraph
containing all the arcs of length $1,2,\ldots,k$, plus some arcs of length $k+1$ if $r>0$.
\end{proposition}
\begin{proof}We distinguish three cases according to the value of $p$.

\textbf{Case 1.} If $p$ is small, that is such that the tournament $T_p$ loads each arc at most $C$
times, then $\gamma(C,p)=\frac{p(p-1)}{2}$. Let us now see for which values of $p$ this fact holds.

If $p$ is odd, the load of $T_p$ is $\frac{p^2-1}{8}\leq C$. The inequality $p^2-1 \leq 8C$ implies
$p^2-1 \leq 4k(k+1)+8r$, and is satisfied if $p \leq 2k+1$, as $p^2-1 \leq 4k(k+1)$.

If $p$ is even, the load of $T_p$ is $\frac{p^2}{8}+\frac{1+\delta}{2}$, where $\delta = 1$ if
$p\equiv 0 \pmod{4}$ (see proof of Proposition~\ref{lem:W}).

If $p \leq 2k$, it holds $\frac{p^2+8}{8} \leq \frac{4k^2+8}{8}\leq \frac{k(k+1)}{2} \leq C$.

For $p = 2k+2$, it holds $\frac{p^2}{8}+\frac{1+\delta}{2}= \frac{k^2}{2}+ k+1+\frac{\delta}{2}
\leq \frac{k^2+k}{2}+r = C$ if and only if $r \geq \frac{k+2+\delta}{2}$, with $\delta =1 $ if $p
\equiv 0 \pmod{4}$, that is, if $k$ is odd. Therefore, the condition is satisfied if $r \geq
\frac{k+2}{2}$.

In the next two cases, we provide first a lower bound on $\gamma(C,p)$, and then we prove a
matching upper bound.

\textbf{Case 2.} If $p = 2k +2$ and $1 \leq r < \frac{k+2}{2}$, a solution is obtained by taking
all the arcs of length $1,2,\ldots,k\left(=\frac{p-2}{2}\right)$ --~giving a load of
$\frac{k(k+1)}{2}$~-- plus $2r-1$ arcs of length $\frac{p}{2}$. For example, we can take the arcs
$\left(i,i+\frac{p}{2}\right)$ for $i = 0,2,\ldots,2r-2 \left( < \frac{p}{2}\right)$ and the arcs
$\left(i,i-\frac{p}{2}\right)$ for $i=1,3,\ldots,2r-3$. The load due to these arcs is at most $r$.
Therefore, in this case $\gamma(C,p) \geq kp + 2r-1$.

\textbf{Case 3.} If $p > 2k+2$ or $p=2k+2$ and $r=0$, a solution is obtained by taking all the arcs
of length $1,2,\ldots,k$ plus $\PartIntInf{\frac{rp}{k+1}}$ arcs of length $k+1$, in such a way
that the load due to these arcs is at most $C$, which is always possible (for example, if $p$ is
prime with $k+1$, we take the requests $((k+1)i,(k+1)(i+1))$ for $0 \leq i \leq
\PartIntInf{\frac{rp}{k+1}}-1$, the indices being taken modulo $p$). Therefore, in this case
\begin{equation}
\label{eq:eqngreater} \gamma(C,p)\geq kp+ \PartIntInf{\frac{rp}{k+1}}\ .
\end{equation}

Let us now turn to upper bounds. Suppose we have a solution with $\gamma$ arcs, $\gamma_i$ being of
length $i$ on $\vec{C}_p$. As each arc of length $i$ loads  $i$ arcs, and the total load of the
arcs of $\vec{C}_p$ is at most $Cp$, we have that
\begin{eqnarray*}
Cp & \geq &\sum_{i=1}^{\infty} i\gamma_i\ \geq\ \sum_{i=1}^{k}
i\gamma_i+(k+1)\left(\gamma-\sum_{i=1}^{k}\gamma_i\right)\\
&= & \sum_{i=1}^{k} ip+ (k+1)(\gamma-kp)+
\sum_{i=1}^{k}\underbrace{(k+1-i)(p-\gamma_i)}_{\geq0}\\
& \geq & \frac{k(k+1)}{2}\cdot p+(k+1)(\gamma-kp).
\end{eqnarray*}
Since $Cp=\frac{k(k+1)}{2} \cdot p+rp$, we obtain $rp\geq(k+1)(\gamma-kp)$, and therefore
\begin{equation}
\label{eq:eqnlower} \gamma(C,p)\ \leq\ kp+\frac{rp}{k+1}\ .
\end{equation}
Combining Equations (\ref{eq:eqngreater}) and (\ref{eq:eqnlower}), we get the result for case 3.
For case 2, i.e. when $p = 2k +2$ and $1 \leq r < \frac{k+2}{2}$, Equation~(\ref{eq:eqnlower})
yields $\gamma(C,p)\leq kp + 2r$. If we have equality, then necessarily $\gamma_i = p$ for
$i=1,\ldots,k$, so we have all arcs of length at most $k$. However, the $2r$ arcs of length at
least $k+1$ induce a load at least $r+1$ on some arc of $\vec{C}_p$, so the total load would be
strictly greater than $C$. Therefore, we have at most $\gamma(C,p) \leq kp + 2r-1$, which gives the
result.
\end{proof}

\begin{proposition}
Let $C= k(k+1)/2 + r$, with $0 \leq  r \leq k$. Then
\begin{equation}
\label{eq:rho_formula} \rho(C) \ = \ k + \frac{r}{k+1}\ .
\end{equation}
\end{proposition}
\begin{proof} In Case 1 of the proof of Proposition~\ref{lem:shortest}, $\rho(C) \leq \frac{p-1}{2}$. If $p \leq
2k+1$, $\rho(C) \leq k$. If $p=2k+2$ and $r \geq \frac{k+2}{2}$, $\rho(C) = k + \frac{1}{2} < k +
\frac{r}{k+1}$. Otherwise, by Equation~(\ref{eq:eqnlower}),
\begin{equation}
\label{eq:rho} \rho(C)\ \leq \ \frac{kp + \frac{rp}{k+1}}{p}\ = \ k + \frac{r}{k+1},
\end{equation}
where $C=\frac{k(k+1)}{2}\ +r$, with $0 \leq r \leq k$. So, in all cases, $\rho(C) \leq k +
\frac{r}{k+1}$. Note that when $p$ is a multiple of $k+1$, Equation~(\ref{eq:eqngreater}) implies
that $\gamma(C,p)\geq kp+ \frac{rp}{k+1}$, and therefore $\rho(C) \geq k + \frac{r}{k+1}$. The
result follows.\end{proof}

Note that in~\cite{WCLF00} the following formula is given, equivalent to
Equation~(\ref{eq:rho_formula}):
\begin{equation}
\label{eq:rho_formula2} \rho(C)\ =\ \frac{C}{k+1}+\frac{k}{2}\ .
\end{equation}

Table~\ref{tab:gamma} shows the parameter $\gamma(C,p)$ for small values of $C$ and $p$, as well as
the parameter $\rho(C)$.
\begin{table}[tbh]
\begin{center}
$$
\begin{array}{|c||c|c|c|c|c|c|c|c|c|c|c|c|c|c|c||c|}
\hline p & 2 & 3 & 4 & 5 & 6 & 7 & 8 & 9 & 10 & 11 & 12 & 13 & 14 & 15 & 16 &
\rho(C)\\
\hline \hline C=1 & 1 & \mathbf{3} & \mathbf{4} & \mathbf{5} & \mathbf{6} & \mathbf{7} & \mathbf{8}
& \mathbf{9} & \mathbf{10} & \mathbf{11} & \mathbf{12} & \mathbf{13} & \mathbf{14} & \mathbf{15}
& \mathbf{16} & 1 \\
\hline C=2 & 1 & 3 & 5 & 7 & 9 & 10 & \mathbf{12} & 13 & \mathbf{15} & 16 & \mathbf{18} & 19 &
\mathbf{21} & 22
& \mathbf{24} & 3/2 \\
\hline C=3 & 1 & 3 & 6 & \mathbf{10} & \mathbf{12} & \mathbf{14} & \mathbf{16} & \mathbf{18} &
\mathbf{20} & \mathbf{22} & \mathbf{24} & \mathbf{26} & \mathbf{28} & \mathbf{30}
& \mathbf{32} & 2 \\
\hline C=4 & 1 & 3 & 6 & 10 & 13 & 16 & 18 & \mathbf{21} & 23 & 25 & \mathbf{28} & 30 & 32 &
\mathbf{35}
& 37 & 7/3 \\
\hline C=5 & 1 & 3 & 6 & 10 & 15 & 18 & 21 & \mathbf{24} & 26 & 29 & \mathbf{32} & 34 & 37 &
\mathbf{40}
& 42 & 8/3 \\
\hline C=6 & 1 & 3 & 6 & 10 & 15 & \mathbf{21} & \mathbf{24} & \mathbf{27} & \mathbf{30} &
\mathbf{33} & \mathbf{36} & \mathbf{39} & \mathbf{42} & \mathbf{45}
& \mathbf{48} & 3 \\
\hline C=7 & 1 & 3 & 6 & 10 & 15 & 21 & 25 & 29 & 32 & 35 & \mathbf{39} & 42 & 45 & 48
& \mathbf{52} & 13/4 \\
\hline C=8 & 1 & 3 & 6 & 10 & 15 & 21 & 27 & 31 & 35 & 38 & \mathbf{42} & 45 & 49 & 52
& \mathbf{56} & 14/4 \\
\hline C=9 & 1 & 3 & 6 & 10 & 15 & 21 & 28 & 33 & 37 & 41 & \mathbf{45} & 48 & 52 & 56
& \mathbf{60} & 15/4 \\
\hline C=10 & 1 & 3 & 6 & 10 & 15 & 21 & 28 & \mathbf{36} & \mathbf{40} & \mathbf{44} & \mathbf{48}
& \mathbf{52} & \mathbf{56} & \mathbf{60}
& \mathbf{64} & 4 \\
\hline
\end{array}
$$
\end{center}
\caption{The parameter $\gamma(C,p)$ for some values of $C$ and $p$, as well as $\rho(C)$. The
\textbf{bold} values achieve $\rho(C)$.\label{tab:gamma}}
\end{table}

\subsection{General Lower Bounds}

By Propositions~\ref{lem:W} and~\ref{lem:shortest}, Equations~(\ref{eq:eqn1}), (\ref{eq:eqn2}), and
(\ref{eq:eqn3}) become

\begin{eqnarray}
 \label{eq:eqn4} A&= &\sum_{p=2}^N pa_p\\
 \label{eq:eqn5} \sum_{p=2}^N a_p&\geq&\PartIntSup{\frac{N^2+\alpha}{8C}}\mbox{, where }\alpha=
\left\{\begin{array}{cl}
-1&\mbox{, if $N$ is odd}\\
4&\mbox{, if $N\equiv 2 \pmod{4}$}\\
8&\mbox{, if $N\equiv 0 \pmod{4}$}\\
\end{array}\right.\\
  \label{eq:eqn6} \sum_{p=2}^{N}a_p\gamma(C,p)&\geq&\frac{N(N-1)}{2}
\end{eqnarray}

We are ready to prove the general lower bound on the number of ADMs used by any solution.
\begin{theorem}[General Lower Bound]
\label{theo:lowerbound} Let $C=\frac{k(k+1)}{2}\ +r$, with $0 \leq r \leq k$. The number of ADMs
required in a bidirectional ring with $N$ nodes and grooming factor $C$ satisfies
\begin{equation}
\label{eq:lowerbound} A(C,N)\ \geq\ \PartIntSup{\frac{N(N-1)}{2 \cdot \rho(C)}}\ =\
\PartIntSup{\frac{N(N-1)}{2}\frac{k+1}{k(k+1)+r}}.
\end{equation}
\end{theorem}
\begin{proof}
Using Equations~(\ref{eq:eqn4}) and (\ref{eq:eqn6}), and the definition of $\rho(C)$, we get that
the number $A$ of ADMs used by any solution satisfies
$$
\frac{N(N-1)}{2}\ \leq\ \sum_{p=2}^{N}a_p \cdot \gamma(C,p)\ =\ \sum_{p=2}^{N}p \cdot a_p \cdot
\rho(C)\ =\ \rho(C)\cdot A.
$$
From the above equation and using Equation~(\ref{eq:rho}), we get
$$
A\ \geq\ \PartIntSup{\frac{N(N-1)}{2 \cdot \rho(C)}}\ =\
\PartIntSup{\frac{N(N-1)}{2}\frac{k+1}{k(k+1)+r}}.
$$\end{proof}

To achieve the lower bound of Theorem~\ref{theo:lowerbound}, the only possibility is to use graphs
on $p$ vertices with $\gamma(C,p)$ arcs. The \textbf{bold} values in Table~\ref{tab:gamma} achieve
$\rho(C)$, and therefore the subgraphs corresponding to those values (which exist by
Proposition~\ref{lem:shortest}) are good candidates to construct an optimal partition of the
request graph.


\paragraph{Comparison with existing lower bounds.} In~\cite{art12}
the \textsc{Ring Traffic Grooming} problem in the bidirectional ring is studied. The authors state
a lower bound regardless of routing for a general set of requests. In the particular case of
uniform traffic, they get a lower bound of $\frac{N^2-1}{4\sqrt{2C}}$ (see~\cite[Theorem 1, page
198]{art12}). They indicate in their article that they can improve this bound by a factor of 2 for
all-to-all uniform unitary traffic. We thank T. Chow and P. Lin for sending us the proof of the
following theorem, which is only announced in~\cite{art12}.

\begin{theorem}[\!\cite{chowlinperso,art12}]
\label{teo:chow} If a traffic instance of ring grooming is uniform and unitary, then, regardless of
routing,
$$
A(C,N) \geq \frac{1}{2\sqrt{C}}\sqrt{\frac{N^2(N-1)^2}{2}-N(N-1)}.
$$
\label{theo:improvedchow}
\end{theorem}
The lower bound we obtained in Theorem~\ref{theo:lowerbound} is greater than the bound of
Theorem~\ref{teo:chow}, but it should be observed that we restrict ourselves to shortest path
symmetric routing. Our bound is $\frac{N(N-1)}{2 \rho(C)}$ and the lower bound of
Theorem~\ref{teo:chow} is less than $\frac{N(N-1)}{2 \sqrt{2C}}$. The fact that our bound is better
follows from the fact that $\rho(C) < \sqrt{2C}$. Indeed,
$$
\rho^2(C)\ \leq\ \left( k + \frac{r}{k+1} \right)^2\ =\ k^2 + \frac{2kr}{k+1}+ \frac{r^2}{(k+1)^2}
\ <\ k^2 + 2r +1\ <\ k^2 + k + 2r\ =\ 2C.
$$

\section{Case $C=1$}
\label{sec:C1} For $C=1$, by Proposition~\ref{lem:shortest} $\gamma(1,p) = p$ if $p \geq 2$.
Furthermore, all the directed cycles
 achieve $\rho(1)$ (see Table~\ref{tab:gamma}).

\begin{theorem}
\label{teo:C1}$$ A(1,N)\ =\ \left\{\begin{array}{cl}
\frac{N(N-1)}{2}&\mbox{, if $N$ is odd}\\
\frac{N^2}{2}&\mbox{, if $N$ is even}\\
\end{array}\right.\\
$$
\end{theorem}
\begin{proof}
For $C=1$, the only possible subgraphs involved in the partition of the edges of $T_N$ are cycles
and paths. If only cycles are used, the total number of ADMs is $\frac{N(N-1)}{2}$, which equals
the lower bound of Theorem~\ref{theo:lowerbound}. Each path involved in the partition adds one
unity of cost with respect to $\frac{N(N-1)}{2}$.

If $N=2q+1$ is odd, by~\cite[Theorem 3.3]{BCY03} we know that the arcs of $T_N$ can be covered with
$q$ $\vec{C}_3$'s and $\frac{q(q-1)}{2}$ $\vec{C}_4$'s. The total number of vertices of this
construction is $3q + 2q(q-1) = q(2q+1) = \frac{N(N-1)}{2}$.

If $N$ is even, each vertex must appear with odd degree in at least one subgraph, so the number of
paths in any construction is at least $N/2$. Therefore, the lower bound becomes $\frac{N(N-1)}{2} +
\frac{N}{2} = \frac{N^2}{2}$. By~\cite[Theorem 3.4]{BCY03} the arcs of $T_N$ can be covered with
\begin{itemize}
\item 4 $\vec{C}_3$'s and $2q^2-3$ $\vec{C}_4$'s, if $N = 4q$ with $q>1$;
\item 2 $\vec{C}_3$'s and $2q^2+2q -1$ $\vec{C}_4$'s, if $N =
4q+2$.
\end{itemize}
For $N=4$, we cover $T_4$ with a $\vec{C}_4$ and two arcs. Note that in these constructions, some
arcs are covered more than once. In both cases, the total number of vertices of the construction is
$\frac{N^2}{2}$, hence the lower bound is attained.

Finally, one can check that in the constructions of~\cite{BCY03}, the length of the arcs involved
in the covering of $T_N$ is in all cases bounded above by $\PartIntInf{\frac{N}{2}}$, and therefore
all the cycles induce load 1.\end{proof}

\begin{observation}
For the original problem  with $G= C_N^*$ and $I=K_N^*$, if we apply Theorem~\ref{teo:C1} we get in
the case $N/2$ a value of $N^2$ ADMS; but if we delete the constraint of symmetric routings we get
a value of $N(N-1)/2$ by using~\cite[Theorems 4.1 and 4.2]{BCY03} (however these constructions use
many $K_2$'s).
\end{observation}

\section {Case $C=2$}
\label{sec:C2} When $C=2$ the general lower bound of Theorem~\ref{theo:lowerbound} gives $A(2,N)
\geq \frac{N(N-1)}{3}$. We first improve this bound in Section~\ref{sec:impr2}, and then give
solutions with a good approximation ratio in  Section~\ref{sec:C2constr}.

\subsection {Improved Lower Bounds}
\label{sec:impr2}

For $C=2$, by Proposition~\ref{lem:shortest} $\gamma(2,2) = 1$, $\gamma(2,3) = 3$, $\gamma(2,4) = 5
$ (note that $\gamma(2,4) = 6$ if the routing is not restricted to be symmetric), and $\gamma(2,p)
= \PartIntInf{\frac{3p}{2}}$ for $p \geq 5$. The optimal solutions for $p \geq 4$ even consist of
the $p$ arcs of length 1 $(i,i+1)$ for $0\leq i \leq p-1$, plus the $p/2$ arcs of length 2
$(2i,2i+2)$ for $0\leq i \leq p/2-1$ (in fact, triangles sharing a vertex; see
Figure~\ref{fig:6vertices} for $p=6$). For $p$ odd we have two classes of optimal graphs (see
Figure~\ref{fig:6vertices} for $p=5$).

\begin{figure}[h!]
\begin{center}
\vspace{-0.6cm}
\includegraphics[width=8.2cm]{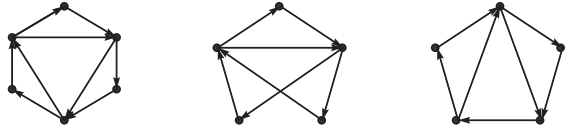}
\vspace{-0.3cm} \caption{\label{fig:6vertices} Some admissible digraphs for $C=2$.} \vspace{-0.3cm}
\end{center}
\end{figure}

Equation~(\ref{eq:eqn6}) becomes in the case $C=2$
$$
\sum_{p=2}^{N}a_p\gamma(2,p) = a_2 + 3 a_3 +5 a_4 + 7 a_5 +9 a_6 + 10 a_7 + 12 a_8 + \ldots \geq
\frac{N(N-1)}{2}.
$$
Therefore,
\begin{eqnarray}
\label{eq:C=2,1}A& =& \sum_{p=2}^N pa_p\ \geq\ \frac{2}{3}\sum_{p=2}^Na_p\gamma(2,p) + \frac{4}{3}
a_2 + a_3 + \frac{2}{3} a_4 + \frac{1}{3} (a_5 +a_7 +a_9 + \ldots)\\
\label{eq:C=2,2}& \geq & \frac{N(N-1)}{3}+ \frac{4}{3} a_2 + a_3 + \frac{2}{3} a_4 + \frac{1}{3}
(a_5 +a_7 +a_9 + \ldots).
\end{eqnarray}
We can already see that the bound $\frac{N(N-1)}{3}$ cannot be attained. Indeed, to reach it we
need to use only graphs with $6,8,10,\ldots$ vertices. But the number of graphs $W$ satisfies, by
Proposition~\ref{lem:W}, $W \geq \frac{N^2-1}{16}$, so $A \geq 6 \frac{N^2-1}{16} >
\frac{N(N-1)}{3}$.

The following proposition gives a lower bound of order $\frac{11}{32}N(N-1)$. Note that $11/32
> 11/33 = 1/3$.

%


\begin{proposition}[Tighter Lower Bound for $C=2$]
\label{prop:lwc2}
\begin{equation}
A(2,N)\geq \PartIntSup{\frac{11 N^2 -8N - 3}{32}} = \PartIntSup{\frac{11}{16}\frac{N(N-1)}{2} +
\frac{3N - 3}{32}}.
\end{equation}
\end{proposition}
\begin{proof}
We can write $A \geq 6(W -a_2 - a_3 - a_4 - a_5) +2a_2 + 3a_3 + 4a_4 + 5a_5$, that is,
\begin{equation}
\label{eq:lwc2_1}A \geq 6W -(4a_2 + 3a_3 + 2a_4 + a_5).
\end{equation}
\indent From Equations~(\ref{eq:C=2,1}) and~(\ref{eq:C=2,2}) we get that
\begin{equation}
\label{eq:lwc2_2}3A \geq N(N-1)+(4a_2 + 3a_3 + 2a_4 + a_5).
\end{equation}
\indent Summing Equations~(\ref{eq:lwc2_1}) and~(\ref{eq:lwc2_2}) gives
\begin{equation}
\label{eq:lwc2_3} 4A \geq 6W +N(N-1).
\end{equation}
\indent By Proposition~\ref{lem:W}, we have that
\begin{equation}
\label{eq:lwc2_4} W \geq \frac{N(N-1)}{16} + \frac{N + \alpha}{16}.
\end{equation}
\indent Combining Equations~(\ref{eq:lwc2_3}) and~(\ref{eq:lwc2_4}) and using that $\alpha \geq -1$
yields
$$
A \geq \frac{11 N(N-1)}{32} + \frac{3N}{32} + \frac{3 \alpha}{32} \geq \frac{11 N^2 -8N - 3}{32}.
$$\end{proof}

\subsection {Upper Bounds}
\label{sec:C2constr}

In this section we build families of solutions for $C=2$. We conjecture that there exists a
decomposition using $A$ vertices with ratio $\frac{A}{\frac{N(N-1)}{2}}$ of order $\frac{11}{16}$,
which would be optimal by Proposition~\ref{prop:lwc2}. For that, we should find some (multipartite)
graphs achieving this ratio. A candidate is $K_{4,4,4}$, which has 48 edges. Unfortunately, we have
not been able to cover it with 33 vertices (which would achieve the optimal ratio) but only with
34, giving a $34/33$-approximation.

For the sake of the presentation, we first present a simple $12/11$-approximation inspired from a
construction of~\cite{BCY03}.

\subsubsection{A $12/11$-approximation} \label{sec:12/11} This construction is
defined recursively. Suppose we have a solution for $N$ vertices using $A_N$ ADMs, with $N=2p$ or
$N=2p+1$. Let the vertex set be labeled $0_A < 1_A < \ldots < (p-1)_A < 0_B < 1_B < \ldots <
(p-1)_B$, plus $\infty$ is $N$ is odd. For $N+2$, we add two vertices $x_A$ and $x_B$ with the
order $x_A < 0_A < 1_A < \ldots < (p-1)_A < x_B < 0_B < 1_B < \ldots < (p-1)_B < \infty$. We use as
subdigraphs those of the solution for $N$ plus the $\PartIntInf{p/2}$ digraphs on the 6 vertices
$x_A,i_A,(i+\PartIntInf{p/2})_A,x_B,i_B,(i+\PartIntInf{p/2})_B$ and the 8 arcs $(x_A,i_A)$,
$(x_A,(i+\PartIntInf{p/2})_A)$, $(i_A,x_B)$, $((i+\PartIntInf{p/2})_A,x_B)$, $(x_B,i_B)$,
$(x_B,(i+\PartIntInf{p/2}_B)$, $(i_B,x_A)$, $((i+\PartIntInf{p/2})_B,x_A)$, for $0 \leq i \leq
\PartIntInf{p/2} - 1$.

If $N=2p$ with $p$ even, there remains uncovered the arc $(x_A,x_B)$.

If $N=2p+1$ with $p$ even, there remain the 3 arcs $(x_A,x_B),(x_B,\infty)$, and $(\infty,x_A)$,
which we cover with the circuit $(x_A,x_B,\infty)$.

If $N=2p$ with $p$ odd, there remain the 5 arcs $(x_A,(p-1)_A), ((p-1)_A,x_B), (x_B,(p-1)_B),
((p-1)_B,x_A)$, and $(x_A,x_B)$, which we cover with a digraph on 4 vertices containing all of
them.

Finally, if $N=2p+1$ with $p$ odd, there remain the 7 arcs $(x_A,(p-1)_A), ((p-1)_A,x_B),
(x_B,(p-1)_B), ((p-1)_B,x_A)$, $(x_A,x_B),(x_B,\infty)$, and $(\infty,x_A)$, which we cover with a
digraph on 5 vertices containing all of them.

One can check that, in all cases, the arcs $(u,v)$ considered satisfy $d_{\vec{C}_n}(u,v) \leq
N/2$.

To compute the number of ADMs of this construction, we have the recurrence relations $A_{4q+2}=
A_{4q}+6q+2$, $A_{4q+4}= A_{4q+2}+6q+4$, $A_{4q+3}= A_{4q+1}+6q+3$, and $A_{4q+5}= A_{4q+3}+6q+5$.
Starting with $A_2=2$ or $A_4=6$ (obtained with the partition with the digraph on 4 vertices formed
by the $C_4$ $(0,1,2,3)$ plus the arc $(0,2)$ and the digraph on 2 vertices $(1,3)$) and $A_3 =3$
or $A_5 = 8$ (obtained with the partition of $T_5$ using the first digraph on 5 vertices of
Figure~\ref{fig:6vertices} and the remaining $T_3$), we get $A_{4q}=6q^2 = \frac{6N^2}{16}$,
$A_{4q+2}=6q^2+6q+ 2 = \frac{6N^2+8}{16}$, $A_{4q+1}=6q^2+2q = \frac{6N^2-4N-2}{16}$, and
$A_{4q+3}=6q^2+8q+3 = \frac{6N^2-4N+6}{16}$.

In all cases, the number of ADMs is of order $\frac{6}{8} \frac{N(N-1)}{2}$, so asymptotically the
ratio between the number of ADMs of this construction and the lower bound of
Proposition~\ref{prop:lwc2} tends to $\frac{6}{8} \frac{16}{11} = \frac{12}{11}$.

\subsubsection{A $34/33$-approximation} \label{sec:34/33}

It will be useful to use the notation $G_5$ and $G_6$ to refer to the digraphs depicted in
Figure~\ref{fig:G5G6}. The key idea of this construction is that an oriented tripartite graph
$K_{4,4,4}$ can be partitioned into admissible subdigraphs for $C=2$ using 34 vertices overall, as
follows.

Let the tripartition classes of the $K_{4,4,4}$ be $\{1_A,1_B,1_C,1_D\}$, $\{2_A,2_B,2_C,2_D\}$,
$\{3_A,3_B,3_C,3_D\}$, and let the vertices be ordered in the ring $1_A < 2_A < 3_A < 1_B < 2_B <
3_B < 1_C < 2_C < 3_C < 1_D < 2_D < 3_D$. The arcs of an oriented $K_{4,4,4}$ can be partitioned
into 4 $G_6$'s with $\{x_1,x_2,x_3,x_4,x_5,x_6\}=\{1_A,2_A,3_B,1_C,2_C,3_D\}$,
$\{1_B,2_B,3_B,1_D,2_D,3_D\}$, $\{1_B,2_C,3_C,1_D,2_A,3_A\}$, and $\{1_A,3_A,2_B,1_C,3_C,2_D\}$,
plus 2 $G_5$'s with $\{x_1,x_2,x_3,x_4,x_5\}=\{3_A,1_C,2_C,1_D,2_D\}$ and $\{3_D,2_A,2_B,1_D,1_C\}$
(see Figure~\ref{fig:G5G6}). The total number of vertices of this partition is $34$.

\begin{figure}[h!]
\begin{center}
\vspace{-.4cm}
\includegraphics[width=10.3cm]{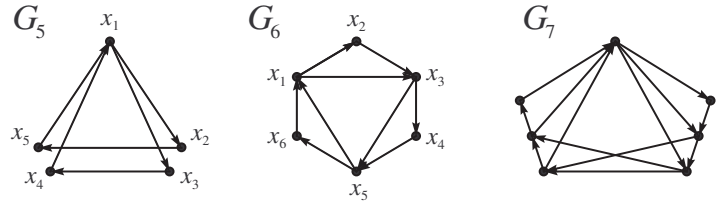}
\vspace{-.5cm} \caption{\label{fig:G5G6} Digraphs $G_5$ and $G_6$ used in the $34/33$-approximation
for $C=2$, and digraph $G_7$ suitable for $C=3$ referred in the proof of
Proposition~\ref{prop:improvedLB3}.} \vspace{-.3cm}
\end{center}
\end{figure}

We are now ready to explain the construction. We take an integer $p \equiv 1\text{ or } 3
\pmod{6}$, hence $K_p$ can be partitioned into triangles. We replace each vertex $i$ of $K_p$ with
4 vertices $i_A,i_B,i_C,i_D$, and order the vertices $1_A < \ldots < p_A < 1_B < 2_B < \ldots < p_B
< 1_C < \ldots < p_C < 1_D < \ldots < p_D$. To a triple $\{i,j,k\}$ corresponding to a triangle of
$K_p$, with $i < j < k$, we associate the decomposition described above of the $K_{4,4,4}$ on
vertices $\{\ell_A,\ell_B,\ell_C,\ell_D : \ell = i,j,k\}$. In this way, $K_{p \times 4}$ can be
partitioned into $\frac{p(p-1)}{6}$ $K_{4,4,4}$'s, or equivalently into $\frac{p(p-1)}{6} \cdot 4$
$G_6$'s and $\frac{p(p-1)}{6} \cdot 2$ $G_5$'s. Overall, we use $\frac{34 p (p-1)}{6}$ vertices.
Each of the subdigraphs of this partition is admissible, as the distance in the ring between the
endpoints of an arc is strictly smaller than $2p$.

To partition an oriented $K_{4p}$, there remain only the $K_4$'s induced inside each class of the
$K_{p \times 4}$. As $A(2,4)=6$, we use $6p$ vertices to cover all the $K_4$'s.

Therefore, if $p \equiv 1\text{ or } 3 \pmod{6}$, an oriented $K_{4p}$ can be partitioned using $6p
+ \frac{34p(p-1)}{6} = \frac{34p^2 + 2p}{6} = \frac{34N^2 + 8N}{96}$ vertices. To decompose
$K_{4p+1}$, we add a vertex $\infty$, and we partition the $p$ $K_5$'s using 8 vertices for each
one of them. Overall, we use $8p + \frac{34p(p-1)}{6} = \frac{34p^2 + 14p}{6} = \frac{34N^2 -12N -
24}{96}$ vertices.

If $p \nequiv 1\text{ or } 3 \pmod{6}$, we introduce dummy vertices to get $p' \equiv 1\text{ or }
3 \pmod{6}$, we do the construction described above, and then we remove the dummy edges and
vertices. It is clear that these dummy vertices add $\mathcal{O}(N)$ vertices to the construction,
hence the coefficient of the term $N^2$ remains the same.

Since $\frac{33N^2 -24N - 9}{96}$ is a lower bound by Proposition~\ref{prop:lwc2}, we get the
following result.

\begin{proposition}
The above construction approximates $A(2,N)$ within a factor $34/33$.
\end{proposition}

\section {Case $C=3$}
\label{sec:C3}

We first provide improved lower bounds for some congruence classes in Section~\ref{sec:C3_LB} and
then we provide constructions in Section~\ref{sec:C3_const}, which are either optimal or
asymptotically optimal.
\subsection{Improved lower Bounds} \label{sec:C3_LB} In this case (see Table~\ref{tab:gamma}) we
have $\gamma(3,2)=1$, $\gamma(3,3)=3$, $\gamma(3,4)=6$, and $\gamma(3,p)=2p$ for $p \geq 5$, so
$\rho(3)=2$. Therefore, by Theorem~\ref{theo:lowerbound}, we get
\begin{proposition}
\label{prop:lowerbound3} $A(3,N)\ \geq\ \frac{N(N-1)}{4}$.
\end{proposition}

By Equations~(\ref{eq:eqn4}) and~(\ref{eq:eqn6}) we have
\begin{eqnarray*}
2A & = & \sum_{p=2}^N 2p a_p\ = \ 4a_2 + 6a_3 + 8a_4 +\sum_{p=5}^N 2p a_p\\
\frac{N(N-1)}{2}& \leq & \sum_{p=2}^N a_p \gamma(3,p)\ =\ a_2 + 3a_3 + 6a_4 + \sum_{p=5}^N 2p a_p
\end{eqnarray*}
So, $$ A \ \geq \ \frac{N(N-1)}{4} + \frac{3}{2}a_2 + \frac{3}{2}a_3 + a_4.
$$
Therefore, if the lower bound is attained, then necessarily $a_2 = a_3 = a_4 = 0$. We will see in
the Section~\ref{sec:C3_const} that this is the case for $N \equiv 1 \text{ or }5 \pmod{12}$, using
optimal digraphs on 5 vertices (namely $T_5$) and on 6 vertices (namely $\vec{K}_{2,2,2}$, see
Figure~\ref{fig:k222}). Optimal graphs are obtained by using arcs of length 1 and 2, so the degree
of any vertex in an optimal subdigraph is 4. That is possible only if the total degree of a vertex,
namely $N-1$, is a multiple of 4. Otherwise, the following proposition shows that the lower bound
of Proposition~\ref{prop:lowerbound3} cannot be attained.

\begin{proposition}
\label{prop:improvedLB3}\textcolor{white}{mh}\\ If $N \equiv 3 \pmod{4}$, $\ \ A(3,N)\ \geq\
\frac{N(N-1)}{4} + \frac{N}{6}\ =\ \frac{3N^2 - N}{12}.$

\noindent If $N \equiv 0 \pmod{2}$, $\ \ A(3,N)\ \geq\ \frac{N(N-1)}{4} + \frac{N}{4}\ =\
\frac{N^2}{4}.$\end{proposition}
\begin{proof}
We use the following observation: If a vertex $x$ has out-degree 3 (resp. in-degree 3) in a digraph
$B_{\omega}$, then its nearest out-neighbor $A_x^+$ (resp. in-neighbor $A_x^-$) has in-degree 1 and
out-degree at most 1 (resp. out-degree 1 and in-degree at most 1). Indeed, suppose $x$ has
out-degree 3, and let $A_x^+,B_x^+,C_x^+$ be the out-neighbors of $x$. Then the load of the arc
entering $A_x^+$ is already 3, so $A_x^+$ has no other in-neighbor than $x$. The load of the arc
leaving $A_x^+$ is already 2, so $A_x^+$ has at most 1 out-neighbor $y$. If $y$ has 2 or more
in-neighbors, then $A_x^+$ is not its nearest one. Hence, to each vertex $x$ of out-degree 3 (resp.
in-degree 3) is associated a distinct vertex $A_x^+$ (resp. $A_x^-$) of degree at most 2.

Consider the digraphs in which a given vertex $x$ appears. Let $\alpha_i^x$ be the number of times
$x$ appears with degree $i$, and let $\alpha_i = \sum_{x}\alpha_i^x$. Vertex $x$ appears in
$\sum_{i} \alpha_{i}^{x}$ digraphs, so
\begin{equation}
\label{eq:improved1} A = \sum_{x} \sum_{i} \alpha_{i}^{x} = \sum_{i} \alpha_{i}.
\end{equation}
As each vertex has degree $N-1$, $N-1 = \sum_{i} i \cdot \alpha_{i}^{x}$, and so
\begin{equation}
\label{eq:improved2} N(N-1) = \sum_{x} \sum_{i} i \cdot \alpha_{i}^{x} = \sum_{i} i \cdot
\alpha_{i}.
\end{equation}
Due to the load constraint, a vertex has out-degree (resp. in-degree) at most 3 in all the digraphs
in which it appears. Therefore, its degree is at most 6, that is, $\alpha_i = 0 $ for $i \geq 7$.
Furthermore, by the above observation if a vertex has degree 6 (resp. 5), to this vertex are
associated 2 vertices (resp. 1 vertex) of degree at most 2, and all these vertices are distinct, so
\begin{equation}
\label{eq:improved3} \alpha_{1} + \alpha_{2} \geq 2 \alpha_{6} + \alpha_{5}.
\end{equation}
Combining Equations~(\ref{eq:improved1}) and~(\ref{eq:improved2}) we get
\begin{equation}
\label{eq:improved4} 4A = N(N-1) + 3\alpha_1 + 2\alpha_2 + \alpha_3 - \alpha_5 - 2\alpha_6.
\end{equation}
We distinguish two cases: $N$ even or $N = 4t+ 3$.

If $N$ is even, $N-1$ is odd and each vertex must appear at least in one $B_{\omega}$ with odd
degree, so
\begin{equation}
\label{eq:improved5} \alpha_1 + \alpha_3 + \alpha_5 \geq N.
\end{equation}
Using Equation~(\ref{eq:improved3}) multiplied by 2 in Equation~(\ref{eq:improved4}) we get $4A
\geq N(N-1) + \alpha_1 + \alpha_3 + \alpha_5 + 2\alpha_6$, so by Equation~(\ref{eq:improved5}), $4A
\geq N(N-1) + N$, as claimed. Note that to obtain equality we need $\alpha_6 = 0$, $\alpha_1 +
\alpha_2 = \alpha_5$, and $\alpha_1 + \alpha_3 + \alpha_5 = N$.

If $N = 4t + 3$, the degree of each vertex satisfies $N-1 \equiv 2 \pmod{4}$, so no vertex can
appear with degree 4 in all the digraphs. Each vertex must appear either at least once with degree
6 or 2, or at least twice with odd degree (for example, 5 and 5, 3 and 3, 1 and 1, or 5 and 1), so
\begin{equation}
\label{eq:improved6} \alpha_2 + \alpha_6 + \frac{1}{2}\left(\alpha_1 + \alpha_3 + \alpha_5\right)
\geq N.
\end{equation}
Equation~(\ref{eq:improved4}) can be rewritten as
\begin{equation}
\label{eq:improved7} 4A = N(N-1) + \frac{2}{3}\left(  \alpha_2 + \alpha_6 +
\frac{1}{2}\left(\alpha_1 + \alpha_3 + \alpha_5\right)  \right) + \frac{4}{3} \left(\alpha_2 +
\alpha_1 - 2 \alpha_6 - \alpha_5 \right) + \frac{2}{3}\alpha_3 + \frac{4}{3}\alpha_1.
\end{equation}
Using Equations~(\ref{eq:improved3}) and~(\ref{eq:improved6}) in Equation~(\ref{eq:improved7})
yields $4A \geq N(N-1) + \frac{2}{3}N + \frac{2}{3}\alpha_3 + \frac{4}{3}\alpha_1$, or $A \geq
\frac{N(N-1)}{4} + \frac{N}{6}$, as claimed. Note that to reach the equality, we need to have
$\alpha_1 = \alpha_3 = 0$, $\alpha_2 = 2\alpha_6 + \alpha_5$ by Equation~(\ref{eq:improved3}), and
$2 \alpha_6 + 2 \alpha_2 + \alpha_5 = 2N$ by Equation~(\ref{eq:improved6}), so $\alpha_2 =
\frac{2N}{3}$, hence an optimal decomposition should use $\frac{N}{3}$ digraphs like the digraph
$G_7$ depicted in Figure~\ref{fig:G5G6}, having 1 vertex of degree 6 and 2 vertices of degree 2.
\end{proof}

\subsection{Constructions}
\label{sec:C3_const}

Our constructions rely on the existence of 3-$GDD$'s, that is, decompositions of complete
multipartite graphs into $K_3$'s. We recall the definition and some basic results below.

\paragraph{Decompositions or complete multipartite graphs into $K_3$'s.} Let
$v_1,v_2,\ldots,v_q$ be non-negative integers; the complete multipartite graph with group sizes
$v_1,v_2,\ldots,v_q$ is defined to be the graph with vertex set $V_1 \cup V_2 \cup \cdots \cup V_q$
where $|V_i|=v_i$, and two vertices $u \in V_i$ and $v \in V_j$ are adjacent if $i \neq j$. Using
terminology of design theory, the graph of type $p_1^{\alpha_1}p_2^{\alpha_2}\ldots p_h^{\alpha_h}$
is the complete multipartite graph with $\alpha_i$ groups of size $p_i$. The existence of a
partition of this multipartite graph into $K_k$'s is equivalent to the existence of a $k$-$GDD$
(\emph{Group Divisible Design}) of type $p_1^{\alpha_1}p_2^{\alpha_2}\ldots p_h^{\alpha_h}$
(see~\cite{CoDi06}). Here we are interested in the existence of 3-$GDD$'s, that is, partitions into
$K_3$'s. When $|V_i|=p$ for all $i$, we denote by $K_{p \times q}$ the multipartite graph of type
$p^q$. Trivial necessary conditions for the existence of a 3-$GDD$ are
\begin{itemize}
\item[\emph{(i)}] the degree of each vertex is even; and
\vspace{-0.2cm}
\item[\emph{(ii)}] the number of edges is a multiple of 3.
\end{itemize}
These conditions are in general sufficient. In particular, the following results will be used
later.

\begin{theorem}[\hspace{0.3pt}\protect\cite{CoDi06}]\label{teo:3GDD}\textcolor[rgb]{1.00,1.00,1.00}{espai.}\\
A 3-$GDD$ of type $2^q$ with $q \geq 3$ exists if and only if $q \equiv 0
\mbox{ or } 1 \pmod{3}$.\\
A 3-$GDD$ of type $2^{q-1}4$ with $q \geq 4$ exists if and only if $q \equiv 1 \pmod{3}$.\\
A 3-$GDD$ of type $3^{q}$ with $q \geq 3$ exists if and only if $q$ is odd.\\
A 3-$GDD$ of type $3^{q-1}1$ with $q \geq 3$ exists if and only if $q$ is odd.\\
A 3-$GDD$ of type $3^{q-1}5$ with $q \geq 5$ exists if and only if $q$ is odd.\\
A 3-$GDD$ of type $3^{q-1}11$ with $q \geq 7$ exists if and only if $q$ is odd.
\end{theorem}

\paragraph{The basic partition.} In what follows $\vec{K}_{2,2,2}$
will denote the digraph on 6 vertices and 12 arcs depicted in Figure~\ref{fig:k222}. This digraph
can be viewed as being obtained from the $K_3$ $(i,j,k)$ with $i<j<k$ by replacing each vertex $i$
with two vertices $i_A$ and $i_B$ forming an independent set.

\begin{figure}[h!]
\begin{center}
\vspace{-.4cm}
\includegraphics[width=13.5cm]{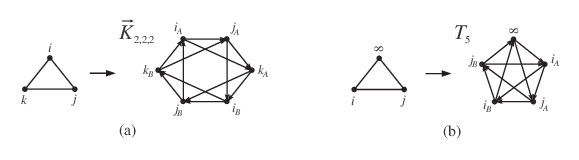}
\vspace{-.5cm} \caption{\label{fig:k222} (a) Digraph $\vec{K}_{2,2,2}$ obtained from $K_3$
$(i,j,k)$, with $i<j<k$; (b) digraph $T_5$ obtained from a $K_3$ of the form $(\infty,i,j)$.}
\vspace{-.3cm}
\end{center}
\end{figure}

Note that $\vec{K}_{2,2,2}$ is an optimal digraph for $C=3$, since it attains the ratio $\rho(3)=2$
(see Table~\ref{tab:gamma}). The idea of the constructions consists in starting from some graph $G$
(mainly a multipartite graph) which can be decomposed into $K_3$'s, replacing each vertex with two
non-adjacent vertices, and then using the following lemma.
\begin{lemma}
\label{lem:C3key} If a graph $G=(V,E)$ with vertex set $\{1,2,\ldots,|V|\}$ can be decomposed into
$h$ $K_3$'s, then the digraph $H$ obtained from $G$ by replacing each vertex $i$ with two
non-adjacent vertices $i_A$ and $i_B$, and where the vertices are ordered
$1_A,2_A,\ldots,|V|_A,1_B,2_B,\ldots,|V|_B$, has a valid decomposition into $\vec{K}_{2,2,2}$'s
with a total of $6h$ vertices.
\end{lemma}
\begin{proof}
To each triangle $(i,j,k)$ with $1\leq i < j < k \leq |V|$ is associated the $\vec{K}_{2,2,2}$ with
vertices $1 \leq i_A < j_A < k_A \leq |V| < i_B < j_B < k_B \leq 2|V|$. To show that the
decomposition is valid for $C=3$, it suffices to show that the distance between the end-vertices of
any arc of any $\vec{K}_{2,2,2}$ is at most $|V|$. That is true for the arcs $(x_A,y_A)$ or
$(x_B,y_B)$ as they satisfy $x<y$, and also for the arcs $(x_A,y_B)$ or $(x_B,y_A)$ as they satisfy
$x>y$ (see Figure~\ref{fig:k222}(a)).
\end{proof}

\paragraph{Some small cases.} We provide here decompositions of some particular small
digraphs that will be used in the constructions of Propositions~\ref{prop:C3upper1}
and~\ref{prop:C3upper2}.
\begin{lemma}
\label{lem:small_cases_C=3} $A(3,5)=5$, $A(3,6) \leq 10$, $A(3,7) \leq 12$, $A(3,8) \leq 18$,
$A(3,9) \leq 21$, $A(3,10) \leq 28$, $A(3,11) \leq 31$, and $A(3,23) \leq 132$.
\end{lemma}
\begin{proof}
Case $N=5$. The decomposition is given in Figure~\ref{fig:k222}(b), and can be viewed as obtained
from the $K_3$ $(\infty,i,j)$ by replacing each of $i,j$ with two vertices.

Case $N=6,7$. The complete graph $K_4$ can be decomposed into one $K_{1,3}$ $(0;\infty,1,2)$ and
one $K_3$ $(\infty,1,2)$. Replace each of the vertices $i,j,k$ with two vertices. The $T_7$ on the
ordered vertices $\infty,0_A,1_A,2_A,0_B,1_B,2_B$ can be partitioned into a $T_5$ on
$\infty,1_A,2_A,1_B,2_B$ ((see Figure~\ref{fig:k222}(b) with $i=1, j=2$)) and the admissible
digraph on 7 vertices and 11 arcs depicted in Figure~\ref{fig:stars2}(b) with $i=0,j=1,k=2$. So we
obtained a valid decomposition using 12 vertices. Deleting vertex $\infty$ yields a decomposition
of $T_5$ with 10 vertices.

Case $N=8,9$. $K_5$ is the union of two $K_3$'s $(\infty,1,3)$, $(0,2,3)$ and a $C_4$
$(\infty,0,1,2)$. Replacing each vertex with two vertices we get a partition of the $T_9$ on the
ordered vertices $\infty,0_A,1_A,2_A,3_A,0_B,1_B,2_B,3_B$. Namely, to the $K_3$ $(\infty,1,3)$ we
associate a $T_5$ on $\infty,1_A,3_A,1_B,3_B$ (see Figure~\ref{fig:k222}(b) with $i=1, j=3$). To
the $K_3$ $(0,2,3)$ we associate a $\vec{K}_{2,2,2}$ on $0_A,2_A,3_A,0_B,2_B,3_B$. To the $C_4$
$(\infty,0,1,2)$ we associate the digraph on 7 vertices of Figure~\ref{fig:stars2}(a) with
$i=0,j=1,k=2$ and the triangle $(1_A,2_A,2_B)$. Therefore, $A(3,9) \leq 21$. Vertex $1_A$ appears
in 3 digraphs, so $A(3,8) \leq 21 - 3 = 18$.

Case $N=10,11$. $K_6$ can be partitioned into 3 $K_3$'s $(\infty,1,3),(\infty,2,4),(0,1,4)$, a star
$K_{1,3}$ $(0;\infty,2,3)$, and a $P_4$ $[1,2,3,4]$. Replacing each vertex with two vertices we get
a partition of the $T_{11}$ on the ordered vertices
$\infty,0_A,1_A,2_A,3_A,4_A,0_B,1_B,2_B,3_B,4_B$ into 2 $T_5$'s on $\infty,1_A,3_A,1_B,3_B$ and
$\infty,2_A,4_A,2_B,4_B$, a $\vec{K}_{2,2,2}$ on $0_A,1_A,4_A,0_B,1_B,4_B$, a digraph on 7 vertices
and 11 arcs depicted in Figure~\ref{fig:stars2}(b) with $i=0,j=2,k=3$, and an admissible digraph on
8 vertices with arcs $(1_A,2_A),(2_A,3_A),(3_A,4_A),(1_B,2_B),(2_B,3_B),(3_B,4_B)$,
$(2_A,1_B),(2_B,1_A),(3_A,2_B),(3_B,2_A),(4_A,3_B),(4_B,3_A)$. Therefore, $A(3,11) \leq 31$, and as
vertex $\infty$ appears in 3 subgraphs, we get $A(3,10)\leq 28$.


Case $N=23$. We decompose $K_{12}$ into 19 $K_3$'s and 3 $K_{1,3}$'s, where vertex $\infty$ appears
in $5$ $K_3$'s and in a star $(i;\infty,j,k)$, the two other stars being of the form
$(i';j'k',\ell')$ with $i'< j' < k' < \ell '$. We obtain a decomposition of $T_{23}$ into 5
$T_5$'s, 14 $\vec{K}_{2,2,2}$'s, 1 digraph of Figure~\ref{fig:stars2}(a), and 2 digraphs of
Figure~\ref{fig:stars2}(c). Thus, $A(3,23) \leq 5\cdot 5 + 14 \cdot 6 + 7 + 8 + 8 = 132$.
\end{proof}
\begin{figure}[t]
\begin{center}
\includegraphics[width=12.cm]{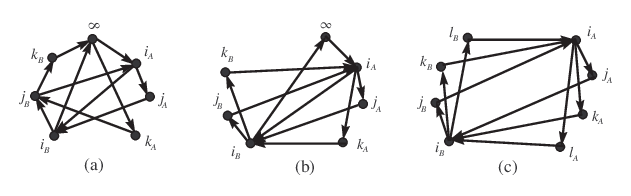}
\caption{(a) Digraph associated to a $C_4$ $(\infty,i,j,k)$. Digraphs associated to stars
($K_{1,3}$'s), with $\infty<i<j<k<\ell$: (b) star of the form $(i;\infty,j,k)$; (c) star of the
form $(i;j,k,\ell)$. \label{fig:stars2}}
\end{center}
\end{figure}

\paragraph{Constructions.} We begin with an optimal partition for
$N \equiv 0,1,4,\text{ or } 5\pmod{12}$, and then we provide near-optimal constructions for the
remaining values.

\begin{proposition}
\label{prop:C3opt} \textcolor{white}{mh}\\
If $N \equiv 0\text{ or } 4\pmod{12}$, $\ \ A(3,N)\ =\ \frac{N^2}{4}.$\\
\noindent If $N \equiv 1\text{ or } 5\pmod{12}$, $\ \ A(3,N)\ =\ \frac{N(N-1)}{4}.$
\end{proposition}
\begin{proof}
The lower bound follows from Propositions~\ref{prop:lowerbound3} and~\ref{prop:improvedLB3}. For
the upper bound, we will apply Lemma~\ref{lem:C3key} with $G = K_{2 \times q}$ (type $2^q$), which
can be decomposed by Theorem~\ref{teo:3GDD} into $\frac{2q(q-1)}{3}$ $K_3$'s if $q \equiv 0 \mbox{
or }1 \pmod{3}$. As $G$ has 2q vertices, the graph $H$ described in Lemma~\ref{lem:C3key} has $4q$
vertices and can be decomposed into admissible $\vec{K}_{2,2,2}$'s. Adding an admissible $T_4$ on
each of the $q$ independent sets of $H$ (of the form $\{i_A,j_A,i_B,j_B\}$ where $\{i,j\}$ is an
independent set of $G$), we get a valid decomposition of $T_{4q}$ into $q$ $T_4$'s and
$\frac{2q(q-1)}{3}$ admissible $\vec{K}_{2,2,2}$'s. So using $A(3,4)=4$, we get $A(3,4q) \leq q
A(3,4) + 4q(q-1) = 4q^2$ for $q \equiv 0 \mbox{ or }1 \pmod{3}$. So $A(3,N) \leq \frac{N^2}{4}$ for
$N \equiv 0 \mbox{ or }4 \pmod{12}$.

For $N=4q+1$, we add to the vertex set of $H$ an extra vertex $\infty$. Adding to the arcs of $H$
the $q$ tournaments $T_5$ built on $\infty,i_A,j_A,i_B,j_B$, where vertices $i,j$ are not adjacent
in $G$, we get a decomposition of $T_{4q+1}$ into $q$ admissible $T_5$'s plus $\frac{2q(q-1)}{3}$
admissible  $\vec{K}_{2,2,2}$'s (the distance being at most $2q-1$ in $H$ and so $2q$ in
$T_{4q+1}$). Using $A(3,5)=5$ (see Lemma~\ref{lem:small_cases_C=3}), we get $A(3,4q+1) \leq q
A(3,5) + 4q(q-1) = 4q^2+q = \frac{(4q+1)4q}{4}$ for $q \equiv 0 \mbox{ or }1 \pmod{3}$. So $A(3,N)
\leq \frac{N(N-1)}{4}$ for $N \equiv 1 \mbox{ or }5 \pmod{12}$.
\end{proof}

We group the non-optimal constructions in Proposition~\ref{prop:C3upper1} and
Proposition~\ref{prop:C3upper2} according to whether they differ from the lower bound by either a
constant or a linear additive term, respectively.

\begin{proposition}
\label{prop:C3upper1} \textcolor{white}{mh}\\
If $N \equiv 8 \pmod{12}$, $\ \ A(3,N)\ \leq\ \frac{N^2}{4}+2.$\\
\noindent If $N \equiv 9 \pmod{12}$, $\ \ A(3,N)\ =\ \frac{N(N-1)}{4}+3.$
\end{proposition}
\begin{proof}
We start from $G$ of type $2^{q-1}4$ with $q \equiv 1 \pmod{3}$, which can be decomposed by
Lemma~\ref{lem:C3key} into $\frac{2(q-1)(q+2)}{3}$ $K_3$'s. As in the proof of
Proposition~\ref{prop:C3opt}, we get a decomposition of $T_{4q+4}$ into $q-1$ $T_4$'s, one $T_8$
and $\frac{2(q-1)(q+2)}{3}$ $\vec{K}_{2,2,2}$'s (indeed, the independent set $V_q$ of $G$ has 4
vertices, so in $H$ it induces an independent set of $8$ vertices). So using $A(3,4)=4$ and $A(3,8)
\leq 18$ (see Lemma~\ref{lem:small_cases_C=3}), we get $A(3,4q+4) \leq (q-1)A(3,4) + A(3,8) +
4(q-1)(q+2) \leq 4q^2 + 8q + 6 = \frac{(4q+4)^2}{4}+2$ for $q \equiv 1 \pmod{3}$, so $A(3,N) \leq
\frac{N^2}{4}+2$ for $N \equiv 8 \pmod{12}$.

Similarly, adding a vertex $\infty$ to $H$ we get a decomposition of $T_{4q+1}$ into $q-1$ $T_5$'s,
one $T_9$ and $h = \frac{2(q-1)(q+2)}{3}$ $K_3$'s. So using $A(3,5)=5$ and $A(3,9) \leq 21$ we get
$A(3,4q+5) \leq (q-1)A(3,5) + A(3,9) + 4(q-1)(q+2) \leq 4q^2 + 9q + 8 = \frac{(4q+5)(4q+4)}{4}+3$
for $q \equiv 1 \pmod{3}$, so $A(3,N) \leq \frac{N(N-1)}{4}+3$ for $N \equiv 9 \pmod{12}$.
\end{proof}

\begin{proposition}
\label{prop:C3upper2} \textcolor{white}{mh}\\
If $N \equiv 2 \pmod{12}$, $\ \ A(3,N)\ \leq\ \frac{N^2}{4} + \frac{N+4}{6}.$\\
\noindent If $N \equiv 3 \pmod{12}$, $\ \ A(3,N)\ \leq\ \frac{N^2+3}{4}.$\\
\noindent If $N \equiv 6 \pmod{12}$, $\ \ A(3,N)\ \leq\ \frac{N^2}{4}+\frac{N}{6}.$\\
\noindent If $N \equiv 7 \pmod{12}$, $\ \ A(3,N)\ \leq\ \frac{N^2-1}{4}.$\\
\noindent If $N \equiv 10 \pmod{12}$, $\ \ A(3,N)\ \leq\ \frac{N^2}{4}+\frac{N+8}{6}.$\\
\noindent If $N \equiv 11 \pmod{12}$, $\ \ A(3,N)\ \leq\ \frac{N^2+3}{4}+\varepsilon,\ $ with
$\varepsilon=1$ for $N=11,35$.
\end{proposition}
\begin{proof}
We use as graph $G$ of Lemma~\ref{lem:C3key} a multipartite graph of type $3^{q-1}u$ with
$3(q-1)+u$ vertices, in order to get a decomposition of $T_{6(q-1)+2u}$ (resp. $T_{6(q-1)+2u+1}$)
into $q-1$ $T_6$'s (resp. $T_7$'s), one $T_{2u}$ (resp. $T_{2u+1}$) and the digraph $H$ itself
decomposed by Lemma~\ref{lem:C3key} into $h = \frac{9(q-1)(q-2)}{6}+u(q-1)$ $\vec{K}_{2,2,2}$'s. We
distinguish several cases according to the value of $u$.

\textbf{Case 1}: $u=1$, $q \geq 3$ odd.

Let $N \equiv 2 \pmod{12}$, $N = 6q-4$. Using that $A(3,2)=2$ and $A(3,6)\leq 10$ we get $A(3,6q-4)
\leq (q-1)A(3,6) + A(3,2) + (q-1)(9q - 12) \leq 9q^2 - 11q + 4 = \frac{(6q-4)^2}{4}+q =
\frac{N^2}{4}+\frac{N+4}{6}$.

Let $N \equiv 3 \pmod{12}$, $N = 6q-3$. Using that $A(3,3)=3$ and $A(3,7)\leq 12$ we get $A(3,6q-3)
\leq (q-1)A(3,7) + A(3,3) + (q-1)(9q - 12) \leq 9q^2 - 9q + 3 = \frac{(6q-3)^2}{4}+\frac{3}{4} =
\frac{N^2 + 3}{4}$.

\textbf{Case 2}: $u=3$, $q \geq 3$ odd.

Let $N \equiv 6 \pmod{12}$, $N = 6q$. Using that $A(3,6)\leq 10$ we get $A(3,6q) \leq q A(3,6) +
9q(q-1) \leq 9q^2 +q = \frac{N^2}{4}+ \frac{N}{6}$.

Let $N \equiv 7 \pmod{12}$, $N = 6q+1$. Using that $A(3,7)\leq 12$ we get $A(3,6q+1) \leq q A(3,7)
+ 9q(q-1) \leq 9q^2 +3q = \frac{N^2-1}{4}$.

\textbf{Case 3}: $u=5$, $q \geq 5$ odd.

Let $N \equiv 10 \pmod{12}$, $N = 6q+4$. Using that $A(3,6)\leq 10$ and $A(3,10)\leq 28$ we get
$A(3,6q+4) \leq (q-1) A(3,6) + A(3,10) + (q-1)(9q+12) \leq 9q^2 + 13q + 6= \frac{(6q+4)^2}{4} +
\frac{6q + 12}{6}= \frac{N^2}{4} + \frac{N+8}{6}$.

Let $N \equiv 11 \pmod{12}$, $N = 6q+5$. Using that $A(3,7)\leq 12$ and $A(3,11)\leq 31$  we get
$A(3,6q+5) \leq (q-1) A(3,7) + A(3,11) + (q-1)(9q+12) \leq 9q^2 + 15q + 7= \frac{N^2+3}{4}$.

For $q = 23$ we have $A(3,23) \leq 132 = \frac{23^2 - 1}{4}$, one less than the value given by the
preceding construction. Using $u = 11$, $q \geq 7$ odd, $N = 6q + 17$, $A(3,7)\leq 12$, and
$A(3,23) \leq 132$ we get $A(3,6q+17) \leq (q-1) A(3,7) + A(3,23) + (q-1)(9q+48) \leq 9q^2 + 51q +
72 = \frac{(6q+17)^2 - 1}{4} = \frac{N^2-1}{4}$. It might be that $A(3,11) \leq 30$, and then the
bound $\frac{N^2 - 1}{4}$ would be also attained for $N = 11$ and $35$.
\end{proof}

\section{Case $C>3$}
\label{sec:C>3}

For $C > 3$, we distinguish two cases according to whether $C$ is of the form $\frac{k(k+1)}{2}$ or
not. We focus on those cases in Sections~\ref{sec:C_not_SUM} and~\ref{sec:C_SUM}.

\subsection{$C$ not of the form $k(k+1)/2$}
\label{sec:C_not_SUM}

If $C$ is not of the form $\frac{k(k+1)}{2}$, we can improve the lower bound of
Theorem~\ref{theo:lowerbound}, as we did for $C=2$ in Proposition~\ref{prop:lwc2}. We provide the
details for $C=4$ and sketch the ideas for $C=5$, that show how to improve the lower bound for any
value of $C$ not of the form $k(k+1)/2$.
\begin{proposition}
\label{prop:C4improved}$$A(4,N)\ \geq\ \frac{7}{32}N(N-1)\ =\ \left( \frac{3}{14} +
\frac{1}{224}\right) N(N-1).$$
\end{proposition}
\begin{proof}
The values of $\gamma(4,p)$ are given in Table~\ref{tab:gamma}, so Equation~(\ref{eq:C=2,1})
becomes in the case $C=4$
\begin{equation}
\label{eq:C=4,1} A = \sum_{p=2}^N pa_p\ \geq \ \frac{3}{7}\sum_{p=2}^Na_p\gamma(4,p) + \frac{11}{7}
a_2 + \frac{12}{7}a_3 + \frac{10}{7} a_4 + \frac{5}{7}a_5 + \frac{3}{7}a_6 + \frac{1}{7} (a_7 +
2a_8 + a_{10} + 2a_{11}+ a_{13} + 2a_{14} +\ldots).
\end{equation}
Using that $\sum_{p=2}^N a_p \gamma(4,p) \geq \frac{N(N-1)}{2}$, Equation~(\ref{eq:C=4,1}) becomes
\begin{equation}
\label{eq:C=4,2} 14A \ \geq \ 3N(N-1) + 22a_2 + 24a_3 + 20a_4 + 10a_5 + 6a_6 + 2a_7 + 4a_8 + \ldots
\end{equation}
On the other hand,
\begin{equation}
\label{eq:C=4,3}A \ \geq \ 9\left(W - \sum_{i=2}^8 a_i\right) +  \sum_{i=2}^8 i \cdot a_i \ = \ 9W
- 7a_2 - 6a_3 - 5a_4 - 4a_5 - 3a_6 - 2a_7 - a_8.
\end{equation}
Summing Equations~(\ref{eq:C=4,2}) and~(\ref{eq:C=4,3}) and using that $W \geq \frac{N(N-1)}{32} +
\frac{N-1}{32}$ by Proposition~\ref{lem:W} yields
$$
15A \ \geq \ \frac{105}{32}N(N-1) + \frac{9}{32}(N-1),\ \text{ and therefore }\ A \ \geq \
\frac{7}{32}N(N-1) + \frac{3}{160}(N-1).
$$\end{proof}
For $C=5$, a similar computation with $\rho(5) = 8/3$ gives
\begin{eqnarray}
\label{eq:C=5,1}8A \ \geq \ \frac{3}{2}N(N-1) + 13a_2 + 15a_3 + 14a_4 + 10a_5 + 3a_6 + 2a_7 + a_8.\\
\label{eq:C=5,2}A \ \geq \ 9W  - 7a_2 - 6a_3 - 5a_4 - 4a_5 - 3a_6 - 2a_7 - a_8.
\end{eqnarray}
So again, Summing Equations~(\ref{eq:C=5,1}) and~(\ref{eq:C=5,2}) and using that $W \geq
\frac{N(N-1)}{40}+ \frac{N-1}{40}$ by Proposition~\ref{lem:W} yields
$$
A\ \geq \ \frac{N(N-1)}{6} + \frac{N(N-1)}{40}+ \frac{N-1}{40}\ =\ \frac{23}{120}N(N-1)+
\frac{N-1}{40}\ =\ \left( \frac{3}{16} + \frac{1}{240}\right)N(N-1)+ \frac{N-1}{40}.
$$

\subsection{$C$ of the form $k(k+1)/2$}
\label{sec:C_SUM}

For $C=\frac{k(k+1)}{2}$ the lower bound of Theorem~\ref{theo:lowerbound} can be attained,
according to the existence of a type of $k$-\emph{GDD}, called \emph{Balanced Incomplete Block
Design} (\emph{BIBD}). A $(v,k,1)$-\emph{BIBD} consists simply of a partition of $K_v$ into
$K_k$'s.

\begin{theorem}
\label{teo:BIBD}If there exists a $(k+1)$-$GDD$ of type $k^q$ (that is, a decomposition of $K_{k
\times q}$ into $K_{k+1}$'s), then there exists an optimal admissible partition of $T_{2kq+1}$ for
$C= \frac{k(k+1)}{2}$ with $\frac{N(N-1)}{2k}$ ADMs.
\end{theorem}
\begin{proof}
The lower bounds follows from Theorem~\ref{theo:lowerbound}. For the upper bound, as we did in
Proposition~\ref{prop:C3opt} (case $k=2$, $C=2$), we replace each vertex $i$ of $K_{k \times q}$
with two vertices $i_A$ and $i_B$, and add a new vertex $\infty$. We label the vertices of the
obtained $T_{2kq+1}$ with $\infty,1_A,\ldots,(kq)_A,1_B,\ldots,(kq)_B$. To each $K_{k+1}$ of the
decomposition of $K_{k \times q}$ we associate a $T_{2 \times (k+1)}$, which is an optimal digraph
for $C= \frac{k(k+1)}{2}$ with $2(k+1)$ vertices and $2k(k+1)$ edges, hence attaining $\rho(C)=k$.
So adding vertex $\infty$ to the stable sets of size $2k$ we obtain a decomposition of $T_{2kq+1}$
into $q$ $T_{2k+1}$'s (which are also optimal) and $T_{2 \times (k+1)}$'s.

If $K_{k \times q}$ is decomposable into $K_{k+1}$'s, the number of $K_{k+1}$'s (and so the number
of $T_{2 \times (k+1)}$'s) is $\frac{kq(q-1)}{k+1}$. Therefore the total number of ADMs is
$q(2k+1)+ 2kq(q-1)=\frac{(2kq+1)2kq}{2k}=\frac{N(N-1)}{2k}$.
\end{proof}
Note that a decomposition of $K_{k \times q}$ into $K_{k+1}$'s is equivalent to a decomposition of
$K_{kq+1}$ into $K_{k+1}$'s by adding a new vertex $\infty$, that is, a $(kq+1,k+1,1)$-\emph{BIBD}.
In particular, such designs are known to exist if $N$ is large enough and $(kq+1)kq \equiv 0
\pmod{k(k+1)}$~\cite{CoDi06}. For example, for $k=3$ and $q \equiv 0 \text{ or }1 \pmod{4}$, or
$k=4$ and $q \equiv 0 \text{ or }1 \pmod{5}$.
\begin{corollary}
\label{cor:BIBD}\textcolor[rgb]{1.00,1.00,1.00}{mh}\\
If $C=6$ and $N \equiv 1 \text{ or }7 \pmod{24}$, $A(6,N)=\frac{N(N-1)}{6}$.\\
If $C=10$ and $N \equiv 1 \text{ or }9 \pmod{40}$, $A(10,N)=\frac{N(N-1)}{8}$.
\end{corollary}
\begin{corollary}
\label{cor:BIBD2}For $C \in \{15,21,28,36\}$, there exists a small set of values of $N$ for which
the existence of a BIBD remains undecided (179 values overall, see~\cite[pages 73-74]{CoDi06}). For
the values of $N$ different from these undecided
BIBDs, the following results apply.\\
If $C=15$ and $N \equiv 1 \text{ or }11 \pmod{30}$, $A(15,N)=\frac{N(N-1)}{10}$ .\\
If $C=21$ and $N \equiv 1 \text{ or }13 \pmod{84}$, $A(21,N)=\frac{N(N-1)}{12}$.\\
If $C=28$ and $N \equiv 1 \text{ or }15 \pmod{112}$, $A(28,N)=\frac{N(N-1)}{14}$.\\
If $C=36$ and $N \equiv 1 \text{ or }17 \pmod{144}$, $A(36,N)=\frac{N(N-1)}{16}$.
\end{corollary}
Wilson proved~\cite{Wilson} that for $v$ large enough, $K_v$ can be decomposed into subgraphs
isomorphic to any given graph $G$, if the trivial necessary conditions about the degree and the
number of edges are satisfied. Thus, we can assure that optimal constructions exist when
$C=\frac{k(k+1)}{2}$ for all $k > 0$.
\begin{corollary}
If $C=\frac{k(k+1)}{2}$, then $A(C,N) = \frac{N(N-1)}{2k}$ for $N \equiv 1 \text{ or } 2k+1
\pmod{4C}$ large enough.
\end{corollary}

We can also use decompositions of $K_{p \times q}$ into $K_{k+1}$'s to get constructions
asymptotically optimal, but not attaining the lower bound like for $C=3$. For instance, for $C=6$
the proof of Theorem~\ref{teo:BIBD} gives (without adding the vertex $\infty$) that for $q \equiv 0
\text{ or }1 \pmod{4}$ and $N \equiv 0 \text{ or }6 \pmod{24}$,
$$
A(6,6q) \ \leq\ q A(6,6) + 6 q(q-1)\ = \ 6q^2 \ = \ \frac{N^2}{6}.
$$
That might be an optimal value if we could improve the lower bound for $C=6$ as we did for $C=3$ in
Proposition~\ref{prop:improvedLB3}, but the calculations become considerably more complicated.

\begin{corollary}
\label{cor:sandwich}\textcolor[rgb]{1.00,1.00,1.00}{mh}\\
For $N \equiv 0 \text{ or }6 \pmod{24}$, $\frac{N(N-1)}{6} \ \leq\ A(6,N)\ \leq \ \frac{N^2}{6}$.\\
For $N \equiv 0 \text{ or }8 \pmod{40}$, $\frac{N(N-1)}{8} \ \leq\ A(10,N)\ \leq \ \frac{N^2}{8}$.
\end{corollary}
For a general $C$ of the form $C = \frac{k(k+1)}{2}$, the improved lower bound one could expect is
$\frac{N^2}{2k}$.

Finally, it is worth to mention here the constructions given in~\cite{CoWa01} for $C=8$. Namely,
in~\cite[Corollary 5]{CoWa01} the authors provide a construction that uses asymptotically
$\frac{N^2}{2}\frac{5}{16}$ ADMs, using the so called \emph{primitive rings}. This construction,
according to the lower bound of Theorem~\ref{theo:lowerbound}, constitutes a
$\frac{35}{32}$-approximation for $C=8$. Note that the construction for $C=6$ given in
Corollary~\ref{cor:BIBD} uses asymptotically $\frac{N^2}{2}\frac{1}{3} = \frac{N^2}{2}\frac{5}{15}$
ADMs, which is already very close to the value obtained in~\cite{CoWa01} for $C=8$, so it seems
natural to suspect that there is enough room for improvement over the constructions
of~\cite{CoWa01}.


%


\section{Unidirectional or Bidirectional Rings?}
\label{sec:comp}

This section is devoted to compare unidirectional and bidirectional rings in terms of minimizing
electronics cost, when these rings are used in a WDM network with traffic grooming and all-to-all
requests.

For bidirectional rings, Theorem~\ref{theo:lowerbound} gives the following lower bound by
multiplying by 2 the value, in order to take into account requests both clockwise and
counterclockwise.
$$
\textsc{LB}_{\mbox{bi}}(C,N)\ =\ \frac{N(N-1)}{2}\cdot \frac{2}{\rho(C)}\ ,
$$
where $\rho(C) = k + \frac{r}{k+1}$ for $C = \frac{k(k+1)}{2} + r$ with $0 \leq r \leq k$.

In~\cite{BeCo03} the following general lower bound was given for unidirectional rings.
$$
\textsc{LB}_{\mbox{uni}}(C,N)\ =\ \frac{N(N-1)}{2}\cdot \frac{1}{\eta(C)}\ ,
$$
$$
\mbox{where }\eta(C)= \left\{\begin{array}{cl} \frac{k}{2}\ ,
&\mbox{if }C = \frac{k(k+1)}{2} + r\ \mbox{ and }\ 0 \leq r \leq \frac{k}{2}\\
 \frac{C}{k+2}\ ,&\mbox{if }C = \frac{k(k+1)}{2} + r\ \mbox{ and }\ \frac{k}{2} \leq r \leq k
\end{array}\right.
$$

Note that for $C= \frac{k(k+1)}{2}$ (that is, for $r=0$) the bounds are equal. In general, we have
$$
1 \ \leq\  \frac{\textsc{LB}_{\mbox{uni}}(C,N)}{\textsc{LB}_{\mbox{bi}}(C,N)}\ \leq \ 1
+\frac{1}{2(k+1)}.
$$

Indeed, either $0 \leq r \leq \frac{k}{2}$ and then

$$
\frac{\rho(C)}{2 \eta(C)}\ = \  1 + \frac{r}{k(k+1)}\ \leq \ 1 + \frac{1}{2(k+1)},
$$
or $\frac{k}{2} \leq r \leq k$, and then
$$
\frac{\rho(C)}{2 \eta(C)}\ = \ \frac{(k+2)(k(k+1)+r)}{(k+1)(k(k+1)+2r)}\ = \ 1 +
\frac{k(k+1)-rk}{(k+1)(k(k+1)+2r)}.
$$
Let $r = \frac{k}{2}+r'$, and so $0 \leq r' \leq \frac{k}{2}$. Then
$$
\frac{\rho(C)}{2 \eta(C)}\ = \ 1 + \frac{1}{2(k+1)}\frac{k(k+2)-2r'}{k(k+2)+2r'} \ \leq\ 1 =
\frac{1}{2(k+1)}.
$$

Note that there exist constructions for bidirectional rings with cost strictly smaller than
$\textsc{LB}_{\mbox{uni}}(C,N)$. Indeed, for $C=2$ we presented in Section~\ref{sec:34/33} a
construction using at most $\frac{17}{48}N(N-1)$ ADMs. Taking into account requests in both
directions this construction uses at most $\frac{17}{24}N(N-1)$ ADMs, to be compared with
$\textsc{LB}_{\mbox{uni}}(2,N) = \frac{3}{4}N(N-1) > \frac{17}{24}N(N-1)$.

However, for large $C$ the lower bounds tend to be equal; hence in terms of the number of ADMs
there is no real improvement in using bidirectional rings. The real improvement is more in terms of
the number of used wavelengths (or, equivalently, the load). Indeed, in unidirectional rings this
number is roughly $\frac{N^2}{2C}$ (see for instance~\cite{BeCo03}), which is twice the number in
bidirectional rings (roughly equal to $2\cdot \frac{N^2}{8C}$ by Proposition~\ref{lem:W}).

In summary, bidirectional and unidirectional rings are equivalent in terms of the number of ADMs,
the trade-off being between better bandwidth utilization in bidirectional rings versus simplicity
(and the use of the other ring for fault tolerance) in unidirectional rings.

\section{Conclusions and Further Research}
\label{sec:concl}
In this article we studied the minimization of ADMs in optical WDM bidirectional ring networks
under the assumption of symmetric shortest path routing and all-to-all unitary requests. We
precisely formulated the problem in terms of graph decompositions, and stated a general lower bound
for all the values of $C$ and $N$. We then studied extensively the cases $C = 2$ and $C=3$,
providing improved lower bounds, optimal constructions for several infinite families, as well as
asymptotically optimal constructions and approximations. To the best of our knowledge, these are
the first optimal solutions in the literature for traffic grooming in bidirectional rings. We then
study the case $C>3$, focusing specifically on the case $C = k(k+1)/2$ for some $k \geq 1$. We gave
optimal decompositions for several congruence classes of $N$, using the existence of some
combinatorial designs. We concluded with a comparison of the switching cost in unidirectional and
bidirectional WDM rings.

Further research is needed to find new families of optimal solutions for other values of $C$. The
first step should be to improve the general lower bound for other values of $C$, namely, finding a
closed formula. It would be  interesting to consider other kinds of routing in bidirectional rings,
 not necessarily symmetric or using shortest paths. Stating which kind of routing
is the best for each value of $N$ and $C$ would be a nice
result. Finally, studying the traffic grooming problem using graph partitioning tools in other topologies, like trees or hypercubes, would be also interesting.\\


\noindent {\bf Acknowledgement.} We would like to thank D. Coudert, T. Chow, and P. Lin for
insightful discussions.

\bibliographystyle{abbrv}
\addcontentsline{toc}{section}{References}
\bibliography{ADM_bi}

\end{document}